\documentclass[journal,twoside,web]{ieeecolor}

\usepackage{generic}
\usepackage{cite}
\usepackage{amsmath,amssymb,amsfonts}
\usepackage{algorithmic}
\usepackage{graphicx}
\usepackage{textcomp}
\def\BibTeX{{\rm B\kern-.05em{\sc i\kern-.025em b}\kern-.08em
    T\kern-.1667em\lower.7ex\hbox{E}\kern-.125emX}}

\usepackage{savesym}
\savesymbol{AND}

\usepackage[ruled,linesnumbered]{algorithm2e}

\usepackage{textcomp}
\usepackage{xcolor}
\usepackage{mathrsfs}
\usepackage{bbm}
\usepackage{tcolorbox}

\usepackage{caption}
\usepackage{subcaption}

\makeatletter
\let\NAT@parse\undefined
\makeatother
\usepackage[colorlinks = true, linkcolor = blue, citecolor = blue]{hyperref}

\newcommand{\ldef}{:=}
\newcommand{\rdef}{=:}
\newcommand{\Mcal}[1]{\mathcal{#1}}
\newcommand{\tth}{^{\text{th}}}

\newcommand{\relu}[1]{\left[ #1 \right]_+}

\graphicspath{{./figs/}}

\def \real{\mathbb{R}}
\def \integer{\mathbb{Z}}
\def \one{\mathbf{1}}
\def \x{\mathbf{x}}
\def \bx{\overline{\x}}
\def \y{\mathbf{y}}
\def \neigh{\mathcal{N}}
\def \bneigh{\overline{\mathcal{N}}}

\def \lag{L}

\def \prob{\textbf{P}}

\def \graph{\mathcal{G}}

\def \flow{\mathcal{F}}

\def \node{\mathcal{V}}

\def \edg{\mathcal{E}} 
\def \arc{\mathcal{A}}
\def \barc{\overline{\arc}}

\def \simplex{\mathcal{S}}

\def \eqpt{\mathcal{X}}
\def \ne{\mathcal{NE}}

\def \inmat{\mathbf{A}}

\def \pe{p}

\def \lev{u}

\def \d{\mathbf{d}}
\def \dyn{\mathbf{F}}

\def \supp{\mathrm{supp}}

\def \path{\Mcal{P}}

\def \M{\mathcal{M}}
\def \z{\mathbf{z}}

\def \ker{\mathrm{ker}}

\newtheorem{theorem}{Theorem}[section]
\newtheorem{corollary}[theorem]{Corollary}
\newtheorem{definition}[theorem]{Definition}
\newtheorem{lemma}[theorem]{Lemma}
\newtheorem{example}[theorem]{Example}
\newtheorem{remark}[theorem]{Remark}

\newcommand{\thmtitle}[1]{\mbox{}\textit{(#1).}}

\newcommand{\bulletsym}{\hbox{$\bullet$}}
\newcommand{\bulletend}{\relax\ifmmode\else\unskip\hfill\fi\bulletsym}
\newcommand{\squaresym}{\hbox{$\blacksquare$}}
\newcommand{\proofend}{\relax\ifmmode\else\unskip\hfill\fi\squaresym}
\newcommand{\trianglesym}{\hbox{$\blacktriangle$}}
\newcommand{\egend}{\relax\ifmmode\else\unskip\hfill\fi\trianglesym}

\newcommand{\bld}[1]{\mathbf{#1}}

\newenvironment{proofa}
{}{\proofend}

\def \papername{Dynamics of a Stratified Population of Optimum Seeking
  Agents on a Network - Part I : Modeling and Convergence Analysis}
\def \shortpapername{Stratified population dynamics on a network: modeling and convergence analysis}

\begin{document}


\title{\papername} \author{Nirabhra Mandal and Pavankumar Tallapragada
  \IEEEmembership{Member, IEEE}
  \thanks{
    This work was partially supported by Robert Bosch Centre for
    Cyber-Physical Systems, Indian Institute of Science,
    Bengaluru. N. Mandal was supported by a fellowship grant from the
    Centre for Networked Intelligence (a Cisco CSR initiative) of the
    Indian Institute of Science, Bengaluru. } \thanks{N. Mandal is
    with the Department of Electrical Engineering, Indian Institute of
    Science, Bengaluru. P. Tallapragada is with the Department of
    Electrical Engineering and the Robert Bosch Centre for
    Cyber-Physical Systems, Indian Institute of Science,
    Bengaluru. \tt\small \{nirabhram, pavant\}@iisc.ac.in } }

\maketitle

\begin{abstract}
  In this work, we consider a population composed of a continuum of
  agents that seek to maximize a payoff function by moving on a
  network. The nodes in the network may represent physical locations
  or abstract choices. The population is stratified and hence agents
  opting for the same choice may not get the same payoff. In
  particular, we assume payoff functions that model diminishing
  returns, that is, agents in ``newer'' strata of a node receive a
  smaller payoff compared to ``older'' strata. In this first part of
  two-part work, we model the population dynamics under three choice
  revision policies, each having varying levels of coordination ---
  i. no coordination and the agents are selfish, ii. coordination
  among agents in each node and iii. coordination across the entire
  population. To model the case with selfish agents, we generalize the
  Smith dynamics to our setting, where we have a stratified population
  and network constraints. To model nodal coordination, we allow the
  fraction of population in a node, as a whole, to take the `best
  response' to the state of the population in the node's
  neighborhood. For the case of population-wide coordination, we
  explore a dynamics where the population evolves according to
  centralized gradient ascent of the social utility, though
  constrained by the network. In each case, we show that the dynamics
  has existence and uniqueness of solutions and also show that the
  solutions from any initial condition asymptotically converge to the
  set of Nash equilibria.
\end{abstract}


\section{Introduction} \label{sec:intro} %

\IEEEPARstart{A}{} number of
frameworks have emerged over the years to explore large scale \emph{multi-agent systems} including
those with optimum seeking agents. Population games and evolutionary dynamics, along with opinion dynamics and swarm control are just a few of such frameworks which are of interest to the scientific community. In such large scale systems, the evolution of the population as a whole is of more interest than that of specific, individual agents. In this paper, we
seek to model the evolution of a population of optimum seeking and
myopic agents that move on a network under different levels of
coordination to maximize a payoff function.

\subsection{Literature Survey}

A well established framework that explores such large scale systems
composed of strategic optimum seeking agents is that of population
games and evolutionary dynamics~\cite{WHS:2010:pop_games}. Such
frameworks also find application in problems related to distributed
control and formation control
\cite{NQ-COM-JBG-GO-AP-EMN:2017:article_app, JBG-GO-NQ:2016:similar,
  JMP-GDG-NQ-LPG:2020:ifac_formation_control_app}. However,
implementation of state dependent restriction on the available actions
is missing in much of the work in this literature. The role of a
network is explored in the framework of evolutionary dynamics on
graphs \cite{EL-CH-MAN:2005:egt1, KP-MB-JR-LJS:2015:egt2,
  LMH-NC-KH:2011, BA-MAN:2014:games_graph,
  BA-GL-YTC-BF-NM-STY-MAN:2017} which is concerned with a finite
population of agents modeled as nodes with the graph being a
representation of the interactions between different agents. Other
works of literature \cite{JBG-GO-NQ:2016:similar,
  GC-FF-LZ:2020:ifac_community, JBG-HT:2018, JBG-HT:2018:NEG,
  JBG-GO-AP-HT:2020:ne_neg_ifac} consider the nodes of the graph to
represent choices with the state of the population being composed of
the fraction of population choosing a particular
node. 
References \cite{JBG-GO-NQ:2016:similar, JBG-HT:2018, JBG-HT:2018:NEG,
  JBG-GO-AP-HT:2020:ne_neg_ifac} do consider the underlying graph in
full generality but assume that the initial condition and the Nash
Equilibrium can only be in the relative interior of the
$n$-dimensional probability simplex. In all such works of literature,
all the agents in a particular node receive the same payoff.

Yet another related area is
that of swarms and swarm control. Reference \cite{VK-SM:2018:pde}
addresses control related problems for swarms with the aid of a
framework of partial differential equations. They assume the swarm to
be a continuum evolving in a continuous space. Markov chains
\cite{IC-AR:2009:ergo, BA-DSB:2015, SB-SJC-FYH:2013:inhomogeneous} are
also popular among works that consider probabilistic movement of
agents from one cell to another in a discretized space. Reference
\cite{BA-DSB:2015} uses the convergence properties of Markov chains to
design appropriate control actions for the swarm to reach a desired
distribution. Reference \cite{IC-AR:2009:ergo} on the other hand works
on designing the Markov chains so that the swarm converges to the
stationary distribution. 

\subsection{Contributions}

In this first part of our work, we study the evolution of a population
of myopic agents that seek to maximize a payoff function by moving on
a network. We model the population as a continuum of agents, with
different fractions of the population located on different nodes in
the network. The nodes in the network may represent physical locations
or choices, in a more abstract sense, available to the infinitesimal
agents. Unlike most works in literature, we consider a stratified
population where agents choosing the same choice get different payoffs
based on the strata they occupy. We study three different levels of
coordination among the agents with an inherent desire to increase
their
payoff. 
At each time instant, the network imposes constraints on the set of
choices that an agent can revise to. We model the evolution of the
population on the network starting from an arbitrary initial
state. Specifically, we characterize the dynamics, describe the set of
Nash equilibria and demonstrate analytically that for all initial
conditions, the trajectories converge to this set. Compared to
\cite{JBG-GO-NQ:2016:similar, JBG-HT:2018, JBG-HT:2018:NEG,
  JBG-GO-AP-HT:2020:ne_neg_ifac}, in this work, we allow both the
initial condition of the population and the equilibrium of the
dynamics to be present anywhere on the $n$-dimensional probability
simplex. Further, in our setup the set of Nash equilibria need not be
a singleton and depends heavily on the graph structure. These make our
model richer and the analysis significantly more challenging.
Compared to our preliminary work~\cite{NM-PT:2020:ifac},
here we consider arbitrary cumulative payoff functions that are
strictly concave as opposed to quadratic functions and introduce a new
dynamics called the stratified smith dynamics (SSD) that models
selfish behavior. The more general payoff functions make the analysis
significantly more challenging and we thoroughly address the problem
here.


\subsection{Organization}
The rest of the paper is organized as follows. In Section \ref{sec:setup},
we provide the basic framework of the model and setup
the core notation and the overall problem that we address
in this paper. In Section \ref{sec:gen_dyn}, we collect and present a few definitions and results that help us in proving results in the rest of the paper. In Section \ref{sec:SSD}, we deal with the case where the agents in the population are selfish; model the dynamics and analyze its convergence. In Section \ref{sec:NBRD}, we give a dynamics in spirit of best response dynamics in game theory but from the perspective of the population fraction in each node. We then analyze existence, uniqueness and convergence properties of the solutions. In Section \ref{sec:NRPM}, we give a centralized dynamics
under which the population evolves and again analyze properties like existence, uniqueness and convergence of the solutions. 
Finally, in Section \ref{sec:conclusion}, we summarize the paper
and provide directions in which it can be extended further.

In part two of this work\cite{NM-PT:2020:journ2}, we present steady
state analysis of the three dynamics SSD, NBRD and NRPM. In
particular, we provide sufficient conditions on the graph under which
all the three dynamics have a unique equilibrium point and in the case
of a general graph, provide a computationally efficient method to
compute the steady state value of the social utility.

\subsection{Notation and Definitions}

We denote the set of real numbers, the set of non-negative real
numbers and the set of integers with $\real$, $\real_+$ and
$\integer$, respectively. We let $[p,q]_\integer$ be the set of all
integers between $p$ and $q$ (inclusive), \emph{i.e.},
$[p,q]_\integer \ldef \{x \in \integer \,|\, p \leq x \leq
q\}$. $\real^n$ (similarly $\real_+^n$) is the cartesian product of
$\real$ (equivalently $\real_+$) with itself $n$ times. If $\bld{v}$
is a vector in $\real^n$, we denote $\bld{v}_i$ as the $i^\mathrm{th}$
component of $\bld{v}$ and for a vector $v \in \real^n$, we let
$\supp(\bld{v}) \ldef \{i \in [1,n]_\integer \,|\, \bld{v}_i \neq
0\}$. A closed neighborhood around $\bld{v}$ with respect to the
$l$-norm is denoted by
$\Mcal{B}_l(\bld{v},r) \ldef \{\bld{w} \in \real^n \,|\, \|\bld{w} -
\bld{v}\|_l \leq r\}$. 
We let $\one$ be the vector, of appropriate size, with all its
elements equal to 1 and we let $\bld{e}_i$ be the vector, again of
appropriate size, with its $i^\mathrm{th}$ element equal to 1 and 0
for all other elements. The empty set is denoted by $\varnothing$. For
two sets $\mathcal{U}, \mathcal{V} \subset \mathcal{Q}$, the set
subtraction operation is denoted by
$\mathcal{U} \setminus \mathcal{V} = \mathcal{U} \cap \mathcal{V}^c$,
where $\mathcal{V}^c$ is the set complement of $\mathcal{V}$ in
$\mathcal{Q}$. If $\mathcal{Q}$ is an ordered countable set, then
$\mathcal{Q}_i$ denotes the $i^\mathrm{th}$ member of $\mathcal{Q}$
and $|\mathcal{Q}|$ is used to represent the cardinality of
$\mathcal{Q}$.  $\{i,j\}$ is used to denote an unordered pair while
$(i,j)$ is used to denote an ordered pair. For a vector
$\bld{v} \in \real^n$, $\bld{v} \geq 0$ is used to denote term wise
inequalities. For a matrix $\bld{M} \in \real^{n \times m}$, the
$ij$-th element of $\bld{M}$ is denoted by
$[\bld{M}]_{ij}$. 
By $\relu{.} : \real \to \real_+$ we denote the function that is
defined as $\relu{x} \ldef \max \{x,0\}$. For a function
$f(\x) : \real^n \rightarrow \real$, $\nabla_\x f$ is used to denote
the gradient of $f(.)$ with respect to $\x$, \emph{i.e.}, the $j\tth$
element of $\nabla_\x f$ is $\frac{\partial f}{\partial
  \x_j}$. 

\section{Problem setup} \label{sec:setup}

In this paper, we consider a \emph{population} composed of a continuum
of \emph{agents} that seek to maximize their \emph{payoff} by moving
on a network with different levels of coordination. Let $\node$ be a
set of nodes, $\edg \subseteq \node \times \node$ be a set of edges
and $\graph \ldef (\node,\edg)$ be an undirected connected graph that
does not contain self loops or multiple edges between any pair of
nodes. Let $N \ldef |\node|$ be the total number of nodes and
$M \ldef |\edg|$ be the total number of edges in the graph. The nodes
in the network may either represent physical locations or
\emph{choices}, in a more abstract sense, that are available to the
infinitesimal \emph{agents} constituting the population. Let
$\x_i \in [0, 1]$ be the \emph{fraction} of the population in node
$i$, or equivalently making the choice $i$. We assume that the overall
population is fixed and, without loss of generality, assume that
$\sum_{i \in \node} \x_i = 1$.

Let $\pe_i(.) : [0,1] \to \real$ be the function that models the
\emph{cumulative payoff} of the fraction $\x_i$. We assume that the
fraction in each node is stratified and the agents in different strata
of a given node receive different payoffs. Let
$[a,b] \subseteq [0,\x_i]$ be an arbitrary interval. Then the agents
of node $i$ that are in the \emph{strata} $[a,b]$ get an \emph{average
  payoff} of
\begin{equation*}
  \frac{\pe_i(b)-\pe_i(a)}{b-a} .
\end{equation*}
Thus the total or cumulative payoff that agents in node $i$ receive is
$\pe_i(\x_i) - \pe_i(0)$. Notice that for a node $i$ if
$a \in [0,\x_i]$,
\begin{equation*}
  \lev_i(a) \ldef \frac{\mathrm{d}\,\pe_i}{\mathrm{d}\,y}(a) \,,
\end{equation*}
is the rate of change of the cumulative payoff at $a$ and is also the
average payoff that the agents in the strata $[a]$ of node $i$
receive. By \emph{strata} $[a]$ we mean the \emph{infinitesimal
  strata} around $a$ in $[0,\x_i]$. We call $\lev_i(.)$ as the
\emph{payoff density function} of node $i$.
We let $\lev_i(0)$ be the right derivative of $\pe_i(.)$ at
\emph{zero} and $\lev_i(1)$ be the left derivative of $\pe_i(.)$ at
\emph{one}. We let $\lev(.)$ be the vector whose $i\tth$ element is
$\lev_i(.)$. Through out this paper, we make the following assumption.
\begin{itemize}
\item[\textbf{(A1)}] For all $i \in \node$, $\pe_i(.)$ is twice
  continuously differentiable and strictly concave. Hence,
  $\forall i \in \node$, $\lev_i(.)$ is a strictly decreasing
  function.
\end{itemize}

The function,
\begin{equation}
	U(\x) \ldef \sum_{i \in \node} \ [\pe_i(\x_i) - \pe_i(0)] \,,
	\label{eq:social_utility}
\end{equation}
which we call as the \emph{social utility function} is defined as the
sum of the cumulative payoffs of agents in all the nodes. This
represents the aggregate payoff that the population receives as a
whole. Note that $U(.)$ is a strictly concave function and
$\bld{e}^T_i \nabla U (\x) = u_i(\x_i)$, $\forall \x$.

Let $\neigh^i$ be the set of all neighbors of node $i$ in the graph
$\graph$ and let $\bneigh^{\,i}= \neigh^i \cup \{i\}$. Given the
undirected graph $\graph$, let
\begin{equation*}
  \arc \ldef \bigcup_{\{i,j\} \in \edg} \{ (i,j),(j,i) \} .
\end{equation*}
It is easy to see that $|\arc| = 2\,|\edg| = 2M$. The arcs in $\arc$
are useful for denoting the inflow and outflow of the population
fractions between adjacent nodes. Given a configuration $\x$, each
infinitesimal agent in a node $i$ may be able to increase its payoff
by moving to its neighboring nodes $\neigh^i$. We are interested in
how the population as a whole evolves under different levels of
coordination.

When an infinitesimal agent in node $i \in \node$ decides to switch to
a node $j \in \neigh^i$, it gets to enter the \emph{newest strata}
$[\x_j]$ of node $j$. We impose this restriction in order to model
diminishing returns, \emph{i.e.} newer agents get lower payoff than
older ones.

\begin{remark}\thmtitle{Diminishing returns} 
  As $\lev_i(.)$'s are strictly decreasing functions, the agents in a
  \emph{newer strata} ($[a]$ for higher $a$) get lower payoffs than
  the ones in an \emph{older strata} ($[a]$ for a lower $a$). Since we
  restrict the incoming agents to enter a node at the newest strata,
  they always get a payoff lesser than the older agents already
  residing in the node. \bulletend
	
\end{remark}

In this paper, we consider the class of choice revision dynamics that
can be expressed in the form
\begin{equation}
  \dot{\x}_i = \sum_{j \in \neigh^i} [\delta_{ji}(\x) -
  \delta_{ij}(\x)], \ \forall i \in \node , \label{eq:fbd}
\end{equation}
where $\delta_{ij}(\x) \geq 0$ denotes the \emph{outflow} of the
fraction of population that moves from node $i$ to node
$j \in \neigh^i$ through the arc $(i,j)$ as a function of the
population state $\x$. We additionally impose the condition that
$\delta_{ij}(\x) = 0$ $\forall j \in \neigh^i$ if $\x_i = 0$ in order
to account for the fact that there cannot be any instantaneous outflow
of population from a node if the node is empty.  We call such dynamics
as \emph{flow balanced dynamics}. By choosing different sets of
functions $\delta_{ij}(.)$'s, we can model different dynamics
using~\eqref{eq:fbd}. In this paper, we consider three specific
dynamics with varying degrees of coordination among the agents as
listed below.

\begin{itemize}
\item \emph{Stratified Smith dynamics:} In the first dynamics, we
  assume that each agent is selfish and revises its choice at
  independent and random time instants. We model the evolution of the
  population's choice configuration by extending the standard Smith
  dynamics~\cite{WHS:2010:pop_games} to the case of stratified
  population. In this dynamics, whenever an agent gets an opportunity
  to revise its choice, it does so by comparing its payoff with the
  payoff of the agents at the newest strata in a neighboring node. We
  model and explore this dynamics in Section \ref{sec:SSD}.

\item \emph{Nodal best response dynamics (NBRD):} In Section
  \ref{sec:NBRD}, we assume that the agents coordinate with each other
  at the nodal level. In this case, the dynamics is the result of the
  fraction in each node $i$ redistributing according to the best
  response of the fraction $\x_i$, as a whole, to the current
  configuration $\x$ while assuming that the fractions in the
  neighboring nodes do not change.

\item \emph{Network restricted payoff maximization (NRPM):} In Section
  \ref{sec:NRPM}, we assume that the agents coordinate across the
  entire population in a centralized manner. The population evolves
  according to network restricted gradient ascent of the social
  utility of the entire population.
\end{itemize}

For each of the dynamics, we analyze properties such as existence and
uniqueness of solutions and convergence. In order to reduce repetition
of ideas and analysis, we first discuss and analyze the general flow
balanced dynamics in the following section.

\section{Flow Balanced Dynamics}\label{sec:gen_dyn}

In this section, we discuss the general flow balanced dynamics, which
can be described in terms of \emph{inflows} and outflows of population
between neighboring nodes. We provide sufficient conditions on a
general flow balanced dynamics under which existence and uniqueness of
solutions are guaranteed. Then for a sub-class of flow balanced
dynamics, we also discuss convergence of solutions.

Recall the flow balanced dynamics in~\eqref{eq:fbd} and recall that
$\delta_{ij}(\x)$ denotes the outflow from node $i \in \node$ to node
$j \in \neigh^i$ as a function of the current state $\x$. Thus, the
rate of change of $\x_i$ is given by the inflows from neighboring
nodes $j$ to $i$ minus the outflows from node $i$ to its neighbors. In
the sequel, we omit the argument of the outflow and denote it by
$\delta_{ij}$ wherever there is no confusion.

For the directed graph $\flow \ldef (\node,\arc)$, let
$\inmat \in \real^{N \times 2M}$ be the incidence matrix. In
particular, we can number each arc in $\arc$ and let $\arc_m$ be the
$m\tth$ arc, with $m \in [1\ , \ 2M]_\integer$. If $\arc_m = (i,j)$,
then $[\inmat]_{im} = -1, [\inmat]_{jm} = 1$ and
$[\inmat]_{km} = 0, \,\forall \, k \in \node \backslash
\{i,j\}$. Similarly, we assemble the elements of the set
$\{\delta_{ij}\}_{(i,j) \in \arc}$ into the vector
$\Delta \in \real^{2M}$ as
\begin{equation}\label{eq:Delta}
  \Delta_m \ldef \delta_{ij}, \text{ for } m \in [1\ , \
  2M]_\integer  \text{ s.t. } \arc_m = (i,j) .
\end{equation}
Then, flow balanced dynamics~\eqref{eq:fbd} is concisely expressed as
\begin{equation}
  \dot{\x} = \inmat \,\Delta(\x) \rdef \dyn(\x) .
  \label{eq:gen_dyn}
\end{equation}
It is easy to see that the simplex,
\begin{equation*}
\simplex \ldef \{ \x \in \mathbb{R}^N \,|\,\, \x \geq 0 \,\,,\, \one^T \x = 1 \},
\end{equation*}
which is a compact set, is positively invariant under
\eqref{eq:gen_dyn}. The following lemma, which we prove in
Appendix~\ref{app:aux-results}, gives a sufficient condition under
which $\dyn(.)$ is locally Lipschitz and consequently
\eqref{eq:gen_dyn} has existence and uniqueness of solutions for each
initial condition in $\simplex$.
\begin{lemma}\label{lem:gen_dyn_EU}
  \thmtitle{Existence and uniqueness of solutions for flow balanced
    dynamics} %
  Suppose there exist closed sets
  $\{\Mcal{C}_i \subseteq \real^N_+ \}_{i \in [1,n]_\integer}$ and
  functions
  $\{ \dyn^i(.) : \real^N \to \real^N \}_{i \in [1,n]_\integer}$ such
  that $\bigcup_{i \in [1,n]_\integer} \Mcal{C}_i = \real^N_+$ and
  $\dyn(\x) \big|_{\x \in \Mcal{C}_i} = \dyn^i(\x) $,
  $\forall \, i \in [1,n]_\integer$, with
  $\dyn^i(\x) = \dyn^j(\x) = \dyn(\x)$ for all
  $\x \in \Mcal{C}_i \cap \Mcal{C}_j$. Suppose
  $\forall i \in [1,n]_\integer$, $\dyn^i$ is locally Lipschitz in the
  domain $\Mcal{C}_i$. Then $\dyn(.)$ is Lipschitz in the domain
  $\simplex$. Further, for each initial condition $\x(0) \in \simplex$
  system~\eqref{eq:gen_dyn} has a unique solution
  $\forall \, t \geq 0$.
  
   \bulletend
\end{lemma}

Next we delve deeper into a subset of this general class of
dynamics. The dynamics proposed in Sections \ref{sec:SSD} and
\ref{sec:NBRD} fall into this subclass. Although the dynamics proposed
in Section \ref{sec:NRPM} does not fall in this subclass, it still
belongs to the general class of flow balanced dynamics. We later use
results from this section to prove key results regarding the proposed
dynamics.

\subsection{Strongly Positively Correlated Flow Balanced Dynamics}

We call any flow balanced dynamics that additionally satisfy the
following criterion of \emph{strong positive correlation} as
\emph{strongly positively correlated flow balanced dynamics} or \emph{strongly positively correlated dynamics} for the sake of brevity.  

\begin{definition}\thmtitle{Strong positive correlation}\label{def:PC}
  We say that $\dyn(.) : \real^N \to \real^N$ in \eqref{eq:gen_dyn} is
  strongly positively correlated with $\lev(.)$ if $\,\forall \x \in \simplex$, $\delta_{ij} = 0$, $\forall (i,j) \in \arc$ such that $\lev_i(\x_i) \geq u_j(\x_j)$.  \bulletend
\end{definition}

Definition~\ref{def:PC} allows us to immediately analyze and make
useful conclusions about an arbitrary strongly positively correlated
flow balanced dynamics, such as their convergence properties. Note
that the evolution of the state $\x$ is constrained by the network,
and the dynamics~\eqref{eq:gen_dyn} is closely determined by the
incidence matrix. We first make a simple observation in the following
lemma, which rules out cyclic outflows. We prove the lemma in
Appendix~\ref{app:SPC-proofs}.

\begin{lemma}\label{lem:no_cycle}
  \thmtitle{Strong positive correlation and acyclic flow} %
  Consider the dynamics~\eqref{eq:gen_dyn} and suppose $\dyn(.)$
  satisfies strong positive correlation. For an arbitrary
  $\x \in \simplex$, suppose $\delta_{ij} \geq 0$,
  $\forall (i,j) \in \arc$ and let
  $\tilde{\arc} \ldef \{(i,j) \in \arc \,|\, \delta_{ij} > 0\}$. Then,
  the graph $\tilde{\flow} \ldef (\node,\tilde{\arc})$ is a directed
  acyclic graph. \bulletend
\end{lemma}

Next we characterize the equilibrium set of the
dynamics~\eqref{eq:gen_dyn}. In general, there may be equilibrium
points of the dynamics outside $\simplex$, but we are interested in
the ones that are in the simplex $\simplex$. This equilibrium set is
\begin{equation}
    \eqpt \ldef \{\x \in \simplex \,\,|\,\, \dyn(\x) =
    \bld{0} \} = \{\x \in \simplex \,\,|\,\, \Delta(\x) \in
    \ker(\inmat) \} .
    \label{eq:eqpt}
\end{equation}
Given Lemma~\ref{lem:no_cycle}, we can further refine this set to the
set of all $\x \in \simplex$ such that $\Delta(\x) = \bld{0}$. The
following lemma characterizes an important subset of the set of
equilibrium points of the dynamics, the significance of which is
illustrated in the remark following the lemma. The proof of the result
appears in Appendix~\ref{app:SPC-proofs}.

\begin{lemma} \label{lem:ne_fbd}
\thmtitle{Non-emptiness of equilibrium set for strongly positively correlated dynamics}
  Consider the dynamics~\eqref{eq:gen_dyn} and suppose $\dyn(.)$ is
  strongly positively correlated with $\lev(.)$. Also let $\delta_{ij} = 0$ $\forall j \in \neigh^i$, if $\x_i = 0$. Then, the set
  \begin{equation}
    \begin{split}
      &\ne := \big\{\x\in\simplex \,\,\big |\,\, \lev_i(\x_i)
      \geq \lev_j(\x_j),\\
      & \qquad \qquad \qquad \qquad \qquad \forall \, j \in \neigh^i,
      \forall \, i \in \supp(\x)\big\}\,.
    \end{split}
    \label{eq:Nash_eq}
  \end{equation}
  is non-empty and $\ne \subseteq \eqpt$, with $\eqpt$ as in
  \eqref{eq:eqpt}.  \bulletend
\end{lemma}

\begin{remark}\thmtitle{Nash equilibria}\label{rem:NE}
  The set $\ne$ in \eqref{eq:Nash_eq} characterizes the \emph{Nash
    equilibria} of the population game. If the population
  configuration $\x$ is in $\ne$, then no agent has an incentive to
  unilaterally deviate from its node (or choice). Note that the
  revision choices available to an agent are dependent on its current
  choice. Moreover, notice that the set of Nash equilibria in
  \eqref{eq:Nash_eq} is the same as the one in the case where the
  population fractions are not stratified, that is where $\lev_i(.)$'s
  are the average payoff functions. Also note that the population
  configuration that globally maximizes the social utility always lies
  in $\ne$.\bulletend
\end{remark}

\begin{example}\label{rem:ne_and_neg}
\thmtitle{Continuum of Nash equilibria}
Consider a node set $\node = \{1,2,3\}$ and let the cumulative payoff functions be of the form $\pe_i(y) = -0.5 \, y^2 - a_i \,y$. Thus the payoff density functions are of the form $\lev_i(y) = -y - a_i$. Let $a_1 = a_3 = 0$ and $a_2 = 5$. Consider the graph $\graph^1$ with node set $\node$ and edge set $\edg^1 = \{\{1,2\},\{2,3\}\}$. For this graph, the set of Nash Equilibria is $\ne^1 = \{\sigma[1,0,0]^T + (1-\sigma)[0,0,1]^T \,|\, \sigma \in [0,1]\}$. Now consider the graph $\graph^2$ with the same node set $\node$ but edge set $\edg^2 = \{\{1,2\},\{1,3\},\{2,3\}\}$. For this graph the set of Nash equilibria is $\ne^2 = \{[0.5,0,0.5]\}$ a singleton. Now, both graphs are connected but $\ne^2 \subset \ne^1$. This is because $\forall \bx \in \ne^1$, every path between nodes $1$ and $3$ (which are in $\supp(\bx)$) has node $2 \notin \supp(\bx)$. Thus the framework proposed in this paper is more general than the one considered in \cite{JBG-HT:2018:NEG, JBG-GO-AP-HT:2020:ne_neg_ifac}.
\bulletend
\end{example}

Finally in this section, we give sufficient conditions under which the
population state evolving under any flow balanced dynamics converges
asymptotically to an equilibrium point.

\begin{theorem}\thmtitle{Asymptotic convergence in flow balanced
    dynamics}\label{th:dyn_conv}
  Consider the dynamics \eqref{eq:gen_dyn} with $\delta_{ij} \geq 0$
  for all $(i,j) \in \arc$ and for all $\x \in \simplex$. Also,
  suppose that $\dyn(.)$ is strongly positively correlated with
  $\lev(.)$. Then, $\ne$ is non-empty and $\ne \subseteq \Mcal{X}$,
  the set of equilibrium points of~\eqref{eq:gen_dyn} in
  $\simplex$. Further, $\forall \,\x(0) \in \simplex$, $\x(t)$
  asymptotically converges to $\Mcal{X}$ and $U(\x(t))$ converges to a
  constant.
\end{theorem}

\begin{proof}
  Recall that Lemma \ref{lem:ne_fbd} guarantees that $\ne$ is
  non-empty and that $\ne \subseteq \Mcal{X}$. In order to prove the
  claim about convergence, consider the Lyapunov-like function
  $V(.) \ldef -U(.)$, see~\eqref{eq:social_utility}. Note that as
  $U(.)$ is strictly concave, $V(.)$ is strictly convex. Now, along
  the trajectories of~\eqref{eq:gen_dyn},
\begin{equation}
  \begin{split}
    \dot{V} & = (\nabla V)^T \inmat \, \Delta = - (\nabla U)^T \inmat
    \, \Delta \\
    & = \sum_{(i,j) \in \arc} \delta_{ij} \,[u_i(\x_i) - u_j(\x_j)]
    \,.
  \end{split}
  \label{eq:V_dot}
\end{equation}
This expression is readily obtained from the equivalence
of~\eqref{eq:gen_dyn} and~\eqref{eq:fbd} and after regrouping the
terms. As $\dyn$ in~\eqref{eq:gen_dyn} is strongly positively
correlated with $\lev(.)$, we have $u_i(\x_i) \geq u_j(\x_j)$ implies
$\delta_{ij} = 0$, which is equivalent to $\delta_{ij} > 0$ implies
$u_i(\x_i) < u_j(\x_j)$. The contrapositive has $\delta_{ij} > 0$
rather than $\delta_{ij} \neq 0$ as, again by the assumption of the
theorem, $\delta_{ij} \geq 0$. We thus have $\dot{V} \leq 0$ and since
the simplex $\simplex$ is positively invariant, LaSalle's invariance
principle~\cite{HK:2002:nsc} says that $\x$ asymptotically converges
to the set $\eqpt$. The convergence of $U(\x(t))$ to a constant is
also guaranteed by LaSalle's invariance principle.
\end{proof}

\begin{remark}\thmtitle{Two interpretations of the rate of change
    of Lyapunov-like functions for flow balanced dynamics}
  In \eqref{eq:V_dot}, two different ways of looking at the rate of
  change of Lyapunov-like function $V$ along the trajectories of
  \eqref{eq:gen_dyn} is provided. In the first line, we are directly
  computing the rate of change of cumulative payoff of the fraction in
  each node and summing them up. In the second line, we are instead
  computing the rate of change in the average payoff of the newest
  strata due to each individual inter-nodal outflows and then adding
  them up. This alternate perspective makes the computation and
  bounding of the rate of change of the Lyapunov-like function
  particularly easy. \bulletend
\end{remark}

In this paper, we are interested more in the class of dynamics that
converges to a population state in the set of Nash Equilibria. The
following corollary provides a sufficient condition for that to
happen, namely $\eqpt = \ne$.
\begin{corollary}
  Suppose the hypothesis in Theorem \ref{th:dyn_conv} holds. In
  addition suppose $\dyn(\x) \neq \bld{0}$ for all $\x \notin
  \ne$. Then $\forall \,\x(0) \in \simplex$, $\x(t)$ asymptotically
  converges to $\ne$. \bulletend
\end{corollary}


\section{Stratified Smith Dynamics} \label{sec:SSD}

In this section, we assume that the agents in the population are
selfish and revise their choices independently. The dynamics we
present for this scenario shares the spirit of the standard Smith
dynamics~\cite{WHS:2010:pop_games} but is different due to stratified
population fractions. We first derive the Smith dynamics for the
stratified population setting. Then, we show that the dynamics
satisfies the strong positive correlation property in
Definition~\ref{def:PC}, which then immediately allows us to conclude
about the convergence properties of the dynamics.

Recall that the $\lev_i(.)$ functions are strictly decreasing. This
implies that the agents that are in a newer strata receive lower
payoffs than the ones that are in an older strata. Recall that, when
an infinitesimal agent in node $i \in \node$ decides to switch to a
node $j \in \neigh^i$, it gets to enter the newest strata $[\x_j]$ of
node $j$. Thus, agents in the strata $[y]$ of node $i \in \node$
compare their payoff density, $\lev_i(y)$, with the payoff density,
$\lev_j(\x_j)$, of agents in strata $[\x_j]$ of node $j \in \neigh^i$
while deciding whether to switch to node $j$. The net outflow from a
node $i \in \node$ to a node $j \in \neigh^i$ is a result of the
agents in each population strata $[y]$ of $i$ taking a decision to
switch to $j$ with a probability $\relu{\lev_j(\x_j) - \lev_i(y)}$
(normalized appropriately) as in the Smith dynamics. Thus the net
outflow from $i$ to $j$ (in the expected sense) is
\begin{equation}
  \delta_{ij} = \int_0^{\x_i} \relu{\lev_j(\x_j) - \lev_i(y)}
  \mathrm{d}y = \int_{y_{ij}}^{\x_i} \left[ \lev_j(\x_j) - \lev_i(y) \right]
  \mathrm{d}y ,
  \label{eq:sd_delta_int}
\end{equation}
where 
\begin{equation}
	y_{ij} \ldef 
	\begin{cases}
		0, & \text{if } \lev_i^{-1}(\lev_j(\x_j)) < 0 \\
		\lev_i^{-1}(\lev_j(\x_j)), & \text{if } 0 \leq \lev_i^{-1}(\lev_j(\x_j)) \leq \x_i \\
		\x_i, & \text{if } \lev_i^{-1}(\lev_j(\x_j)) > \x_i \, .
	\end{cases}
	\label{eq:yij}
\end{equation}
In the second equality of \eqref{eq:sd_delta_int}, we have used the
fact that $\lev_i(.)$ is a strictly decreasing function. Note that as
$\lev_i(.)$ is a strictly decreasing continuously differentiable
function, its inverse also exists. So,
\begin{equation}
  	\delta_{ij} = \relu{ \lev_j(\x_j) ( \x_i - y_{ij} ) - ( \pe_i(\x_i) -
    \pe_i(y_{ij}) )} .
  \label{eq:sd_delta_ij}
\end{equation}
From \eqref{eq:sd_delta_ij}, it is clear that $\delta_{ij} \geq 0$,
$\forall (i,j) \in \arc$. Then the dynamics in~\eqref{eq:gen_dyn} with
$\delta_{ij}$ defined in~\eqref{eq:sd_delta_ij} is what we call as the
\emph{stratified Smith dynamics} (SSD).

\begin{remark}\thmtitle{Existence and uniqueness of solutions for SSD}
  Recall that $\forall i \in \node$, $\pe_i(.)$ is a twice
  continuously differentiable function and hence $\lev_i(.)$ is
  strictly decreasing and continuously differentiable. As a result,
  the derivative of $\lev_i(.)$ is not zero at any point in
  [0,1]. Thus, by the inverse function theorem, $\lev^{-1}_i(.)$ is
  also differentiable and locally Lipschitz in [0,1]. Consequently,
  $\inmat \, \Delta(.)$ as a whole, is locally Lipschitz on
  $\simplex$. Then Lemma~\ref{lem:gen_dyn_EU} guarantees the existence
  and uniqueness of solutions of SSD $\forall t \geq 0$ and for each
  initial condition $\x(0) \in \simplex$. \bulletend
\end{remark}

\subsection{Convergence in Stratified Smith Dynamics}

We first show that SSD satisfies Definition~\ref{def:PC} and that the
set of equilibrium points $\eqpt = \ne$ in the following lemma and
subsequently we show that all solutions of SSD converge asymptotically
to the equilibrium set $\eqpt = \ne$. The proof of the following lemma
appears in Appendix~\ref{app:SSD-lemma}.

\begin{lemma}\thmtitle{Strong positive correlation and equilibria of
    SSD} \label{lem:SSD-spc-eqpt}
  Suppose $\forall \, i \in \node$, $\lev_i(.)$ is strictly
  decreasing. Then, SSD, dynamics in \eqref{eq:gen_dyn} with
  $\delta_{ij}$ defined in \eqref{eq:sd_delta_ij}, is strongly
  positively correlated with $\lev(.)$. Further, the set of
  equilibrium points of SSD in $\simplex$, \emph{i.e.} $\eqpt$, is the
  set $\ne$ in \eqref{eq:Nash_eq}.
  \bulletend
\end{lemma}

Now, we can state the result on asymptotic convergence in stratified
Smith dynamics.

\begin{theorem}\thmtitle{SSD asymptotically converges to the Nash
    equilibrium set} \label{th:ssd_conv}
  Let the evolution of $\x$ be governed by SSD, dynamics in
  \eqref{eq:gen_dyn} with $\delta_{ij}$ defined in
  \eqref{eq:sd_delta_ij}, with an initial condition
  $\x(0) \in \simplex$. 
  Then $\x(t)$ converges asymptotically to $\ne$,
  defined in \eqref{eq:Nash_eq}, and $U(\x(t))$ converges to a
  constant.
\end{theorem}

\begin{proof}
  This is a direct consequence of Lemma~\ref{lem:SSD-spc-eqpt} and
  Theorem \ref{th:dyn_conv}.
\end{proof}

\section{Nodal Best Response Dynamics} \label{sec:NBRD}

In this section, we assume that the infinitesimal agents in each node
$i \in \node$ coordinate and seek to maximize the overall payoff that
the fraction $\x_i$ receives by redistributing itself among its
neighbors. Thus, we call this dynamics as \emph{nodal best response
  dynamics} (NBRD). In this section, we present the NBRD dynamics and
analyze its properties such as existence and uniqueness of solutions
and convergence.

Here, like in Section \ref{sec:SSD}, the agents in a node react purely
in response to the current population configuration and in fact assume
that the fractions in other nodes would not revise their
choices. Further, if from a node $i \in \node$ a fraction of
population $\delta_{ij} \leq \x_i$ decides to switch over to node
$j \in \neigh^i$, they assume that they would occupy the strata
$[\x_j,\x_j+\delta_{ij}]$ in node $j$. Under these assumptions, the
fraction $\x_i$ in node $i$ determines the reallocation
$\{\d^*_{ij} \}_{j \in \bneigh^{\,i}}$ of $\x_i$ among the nodes $i$
and its neighbors that maximizes the overall payoff of the fraction
$\x_i$. Then, the outflows $\delta_{ij} = \d^*_{ij}$ for all
$j \in \neigh^i$ and for all $i \in \node$. This behaviour is in the
spirit of best response dynamics \cite{WHS:2010:pop_games} in
evolutionary dynamics but from the fraction $\x_i$'s point of view
rather than from the point of view of infinitesimal agents.

Specifically, the outflows $\delta_{ij}$ from node $i$ are equal to
$\d^*_{ij}$, the optimizer of the problem
\begin{align}
  \label{eq:p_nbrd-prob}
  \prob^{\,i}_2:\\
  \notag \max_{\{\d_{ij} | j \in \bneigh^{\,i}\}}%
  & \sum_{j \in \neigh^i}[\pe_j(\x_j+\d_{ij}) - \pe_j(\x_j)] +
    [\pe_i (\d_{ii}) - \pe_i(0)]\\
  \notag \text{s.t. } & \d_{ii} + \sum_{j \in \neigh^i}
                        \d_{ij} = \x_i, \quad \d_{ij} \geq 0, 
                        \ \forall j \in \bneigh^{\,i}. 
\end{align}
Note that $\pe_j(\x_j+\d_{ij}) - \pe_j(\x_j)$ is the cumulative payoff
received by the agents in strata $[\x_j,\x_j+\d_{ij}]$ of node
$j \in \neigh^i$ if the fraction $\d_{ij}$ chooses to move to node $j$
under the assumption that the agents in node $j$ do not revise their
choice. The non-negativity constraints $\d_{ij} \geq 0$ and
$\d_{ii} \geq 0$ ensure that $\d_{ij}$'s are outflows and $\d_{ii}$ is
the fraction remaining in node $i$. In order to model and analyze the
dynamics of the population in such a scenario, it is useful to analyze
the optimizers of $\prob^{\,i}_2$. The following lemma lists some
important properties of the said optimizers. We provide its proof in
Appendix~\ref{app:NBRD-proofs}.

\begin{lemma}\label{lem:opt_prob2i}\thmtitle{Optimizers of $\prob^{\,i}_2$}
  For all $\x \in \simplex$, the problem $\prob^{\,i}_2$ in
  \eqref{eq:p_nbrd-prob} possesses a unique optimizer,
  $\{\d^*_{ij} \}_{j \in \bneigh^{\,i}}$. Moreover, if $\x_i > 0$
  then there exists a constant $\eta_i$ such that
\begin{subequations}
  \begin{align}
    &\lev_i(\d^*_{ii}) = \eta_i, \ \qquad \quad
      \text{if}\,\,\d^*_{ii} > 0, \label{eq:level_same_own}\\
    &\lev_j(\x_j + \d^*_{ij}) =
      \eta_i, \quad \forall j \in
      \neigh^i \,\,\text{s.t.}\,\,
      \d^*_{ij} > 0, \label{eq:level_same_neigh} \\
    &\lev_j(\x_j) \leq \eta_i, \ \qquad \quad \forall j \in \bneigh^i
      \,\,\text{s.t.}\,\, \d^*_{ij} = 0 \,. \label{eq:level_diff}
  \end{align}
\end{subequations}
\bulletend
\end{lemma}

\begin{remark}\thmtitle{Nodal best response dynamics}
  $\prob^{\,i}_2$ repeated for all $i \in \node$ gives the set
  $\{\d^*_{ij}\}_{(i,j) \in \arc}$ of all outflows on every arc
  $(i,j) \in \arc$, that is,
  \begin{equation}
    \delta_{ij} = \d^*_{ij} , \quad \forall (i,j) \in
    \arc, \label{eq:nbrd-outflows}
  \end{equation}
  where $\{\d^*_{ij}\}_{(i,j) \in \arc}$ is the set of optimizers of
  $\{\prob^{\,i}_2\}_{i \in \node}$. Then, the dynamics given by
  \eqref{eq:gen_dyn} with $\delta_{ij}$'s as
  in~\eqref{eq:nbrd-outflows} is what we refer to as nodal best
  response dynamics (NBRD). Note that if $\x$ is changed, the set of
  optimizers $\{\d^*_{ij}\}_{(i,j) \in \arc}$ and hence $\Delta$ may
  change too. Thus $\Delta(\x)$ is a function of $\x$.  \bulletend
\end{remark}

Application of Lemma~\ref{lem:opt_prob2i} also lets us infer that NBRD
is strongly positively correlated with $\lev(.)$, which we state in
the following lemma, whose proof we provide in
Appendix~\ref{app:NBRD-proofs}.

\begin{lemma}\label{lem:nbrd_str_pos_cor}
  NBRD, dynamics~\eqref{eq:gen_dyn} with $\delta_{ij} = \d^*_{ij}$ for
  all $j \in \bneigh^i$ and for all $i \in \node$, is strongly
  positively correlated with $\lev(.)$. \bulletend
\end{lemma}

Although Lemma \ref{lem:nbrd_str_pos_cor} allows us to directly apply
most of the results in Section~\ref{sec:gen_dyn}, the description of
NBRD through the optimization problem~\eqref{eq:p_nbrd-prob} poses a
few challenges for applying Lemma~\ref{lem:gen_dyn_EU}. Thus, in the
following lemma, we further use Lemma \ref{lem:opt_prob2i} to give an
alternate characterization of the dynamics under specific
scenarios. This characterization allows us to apply
Lemma~\ref{lem:gen_dyn_EU} and establish existence and uniqueness of
solutions for NBRD. The proof of the following lemma appears in
Appendix~\ref{app:NBRD-proofs}.

\begin{lemma}\label{lem:nbrd_alt_dyn}\thmtitle{Computation of
    optimizers of $\prob^{\,i}_2$ under a given flow graph} Let
  $\{\d^*_{ij} \}_{j \in \bneigh^{\,i}}$ be the optimizer of
  $\prob^{\,i}_2$ for a given $\x$ and let $\M^i$ be the support of
  the optimizer. Consider the function $g_{\M^i}(.) : \real \to \real$
  defined as
  \begin{equation}
    g_{\M^i}(y) \ldef \sum_{j \in \M^i} \lev^{-1}_j(y) \, .
    \label{eq:g_func_brd}
  \end{equation}
  Then $\eta_i$ of Lemma \ref{lem:opt_prob2i} is given by
  \begin{equation}
    \eta_i = g_{\M^i}^{-1}\left(\sum_{k \in \M^i \cup \{i\}} \x_k\right).
    \label{eq:eta_brd}
  \end{equation}
  Moreover, $\forall j \in \M^i \setminus \{i\}$
  \begin{subequations}
    \label{eq:delta_brd}
    \begin{equation}
      \d^*_{ij} = \lev^{-1}_j \left( g_{\M^i}^{-1}\left(\sum_{k \in \M^i
            \cup \{i\}} \x_k\right) \right) - \x_j \,,
      \label{eq:delta_ij_brd}
    \end{equation}
    and
    \begin{equation}
      \d^*_{ii} = \lev^{-1}_i \left( g_{\M^i}^{-1}\left(\sum_{k \in \M^i
            \cup \{i\}} \x_k\right) \right), \ \text{ if } i \in \M^i
      .
      \label{eq:delta_ii_brd}
    \end{equation}
  \end{subequations}
  \bulletend
\end{lemma}

From the KKT conditions~\eqref{eq:nbrd_kkt}, notice that the
multipliers $\mu_{ij}^*$ encode the support $\M^i$ while
$\eta_i = \lambda_i^*$. Thus, the characterization of $\{\d^*_{ij}\}$
in Lemma~\ref{lem:nbrd_alt_dyn} is closely related to solving the dual
of problem~$\prob^{\,i}_2$. This alternate characterization of
$\{\d^*_{ij}\}$ through the dual helps us in applying
Lemma~\ref{lem:gen_dyn_EU} in order to guarantee existence and
uniqueness of solutions for NBRD. We now present the main result of
this section where we show asymptotic convergence
of solutions of NBRD to the set $\ne$. We provide the proof the result
in Appendix~\ref{app:NBRD-proofs}.

\begin{theorem}\label{th:NBRD_full}%
  \thmtitle{Existence and uniqueness of solutions for NBRD and
    asymptotic convergence to the Nash equilibrium set} %
  Consider NBRD, the dynamics in \eqref{eq:gen_dyn} with
  $\delta_{ij} = \d^*_{ij}$ $\forall (i,j) \in \arc$, where
  $\{\d^*_{ij}\}_{(i,j) \in \arc}$ is the set of optimizers of the
  problems $\{\prob^{\,i}_2\}_{i \in \node}$. For each
  $\x(0) \in \simplex$, NBRD has a unique solution that exists for
  $\forall \, t \geq 0$. The set of equilibrium points of NBRD in
  $\simplex$ is $\ne$ defined in \eqref{eq:Nash_eq}. Further, if $\x$
  evolves according to NBRD, then as $t \to \infty$, $U(\x(t))$
  converges to a constant and $\x(t)$ approaches $\ne$. \bulletend
\end{theorem}

%
%
%

\section{Network Restricted Payoff Maximization}
\label{sec:NRPM}

In this section, we analyze the dynamics arising out of centralized
network restricted gradient ascent of the social utility, $U(.)$ given
in \eqref{eq:social_utility}. We first analyze the underlying
optimization problem and then formally define the dynamics.
Specifically, we define the dynamics by letting the outflows
$\delta_{ij} = \d_{ij}$, $\forall (i,j) \in \arc$, which come from the
optimizers of the problem $\prob_3$, where
\begin{subequations}
  \label{eq:prob_nrpm}
  \begin{align}
    &\prob_3: \notag\\
    &\max_{\z, \d} \sum_{i \in
      \node}[\pe_i(\z_i ) - \pe_i(0)] \quad = \quad \max_{\z, \d} U(\z)
    \\
    &\text{s.t. } \z_i = \x_i + \sum_{j \in \bneigh^{\,i}} (\d_{ji} -
      \d_{ij}) , \ \forall \, i \in
      \node, \label{eq:nrpm-flow-balance}\\
    &\sum_{j \in \bneigh^{\, i}}\d_{ij}  = \x_i, \ \forall
      \, i \in \node, \label{eq:nrpm-xi-reallocation}\\
    &\d_{ij} \geq 0, \ \forall \, (i,j) \in
      \barc \ldef \arc \cup \{(i,i) \ | \ i \in \node \}
      . \label{eq:nrpm-non-neg}
  \end{align}
\end{subequations}

Note that although $U(\z)$ is strictly concave in $\z$, the cost
function in $\prob_3$ is only concave in $(\z, \d)$. Thus, in general,
there may be more than one optimizer for $\prob_3$. This is unlike the
case for NBRD, where we have unique optimizers for $\prob^{\,i}_2$,
$\forall \, i \in \node$. We show an example in the following.

\begin{example}\thmtitle{Non-unique optimal redistribution choices
    under NRPM} 
  Consider the graph with node set $\node = \{1,2,3,4\}$ and edge set
  $\edg = \{\{1,2\},\{2,3\},\{3,4\},\{4,1\}\}$. Let the cumulative
  payoff functions be of the form $\pe_i(y) = -0.5 \, y^2 - a_i \,
  y$. Thus the payoff density functions are of the form
  $\lev_i(y) = -y - a_i$. This is a uniform water-tank model as described in \cite{NM-PT:2020:ifac}. Let $a_1 = a_3 = 1$, $a_2 = a_4 = 0$ and
  $\x = [0.5,0,0.5,0]^T$. A quick computation reveals that the optimal
  redistributed population state in this case is
  $\z^* = [0,0.5,0,0.5]^T$. This can be done with multiple choices for
  $\d$, two of which are shown in Figure
  \ref{fig:nrpm_opt_diff}. \bulletend
\end{example}

\begin{figure}
	\begin{center}	 
	\begin{tabular}{cc}
		\includegraphics[scale=0.27]{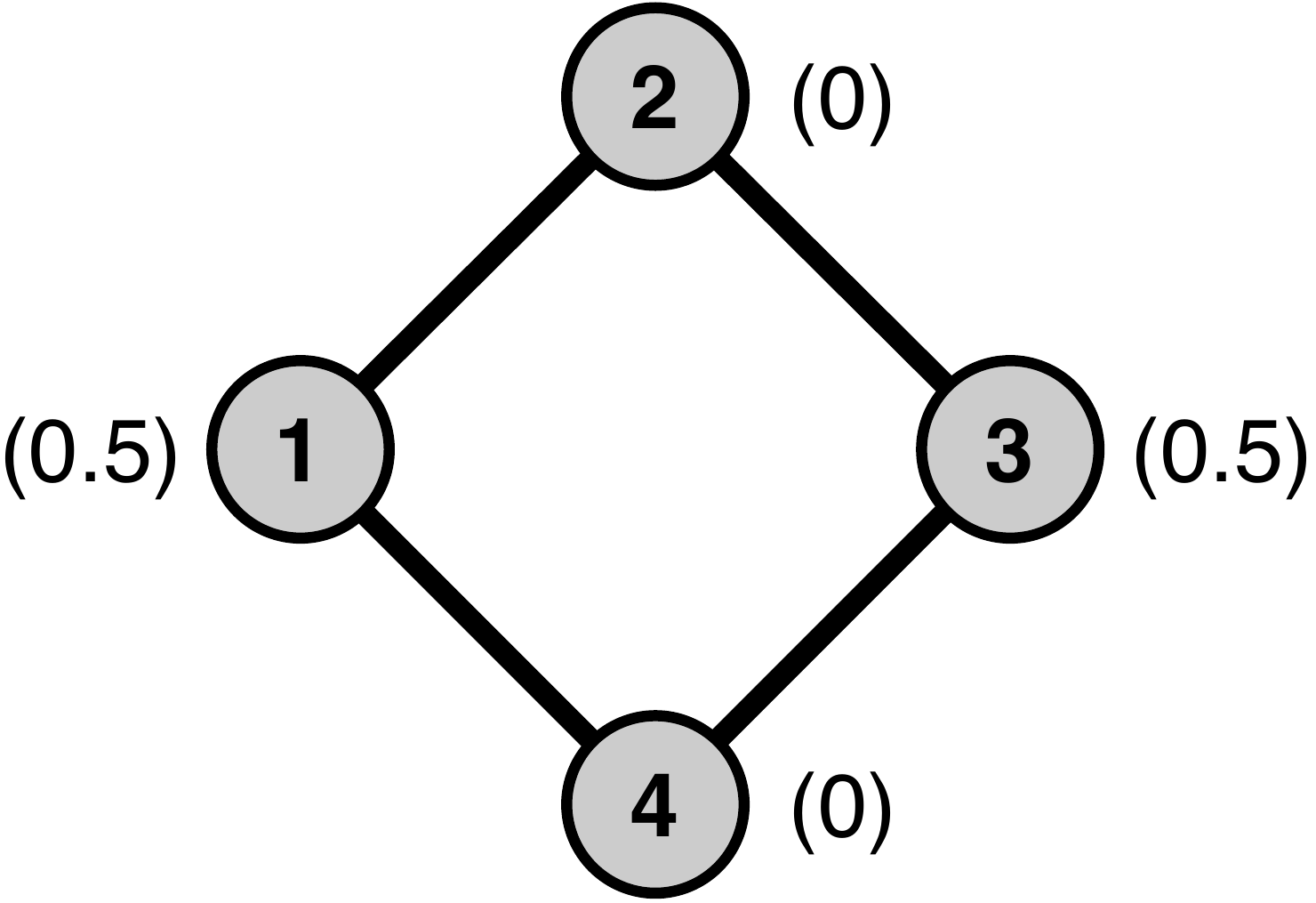} & \includegraphics[scale=0.27]{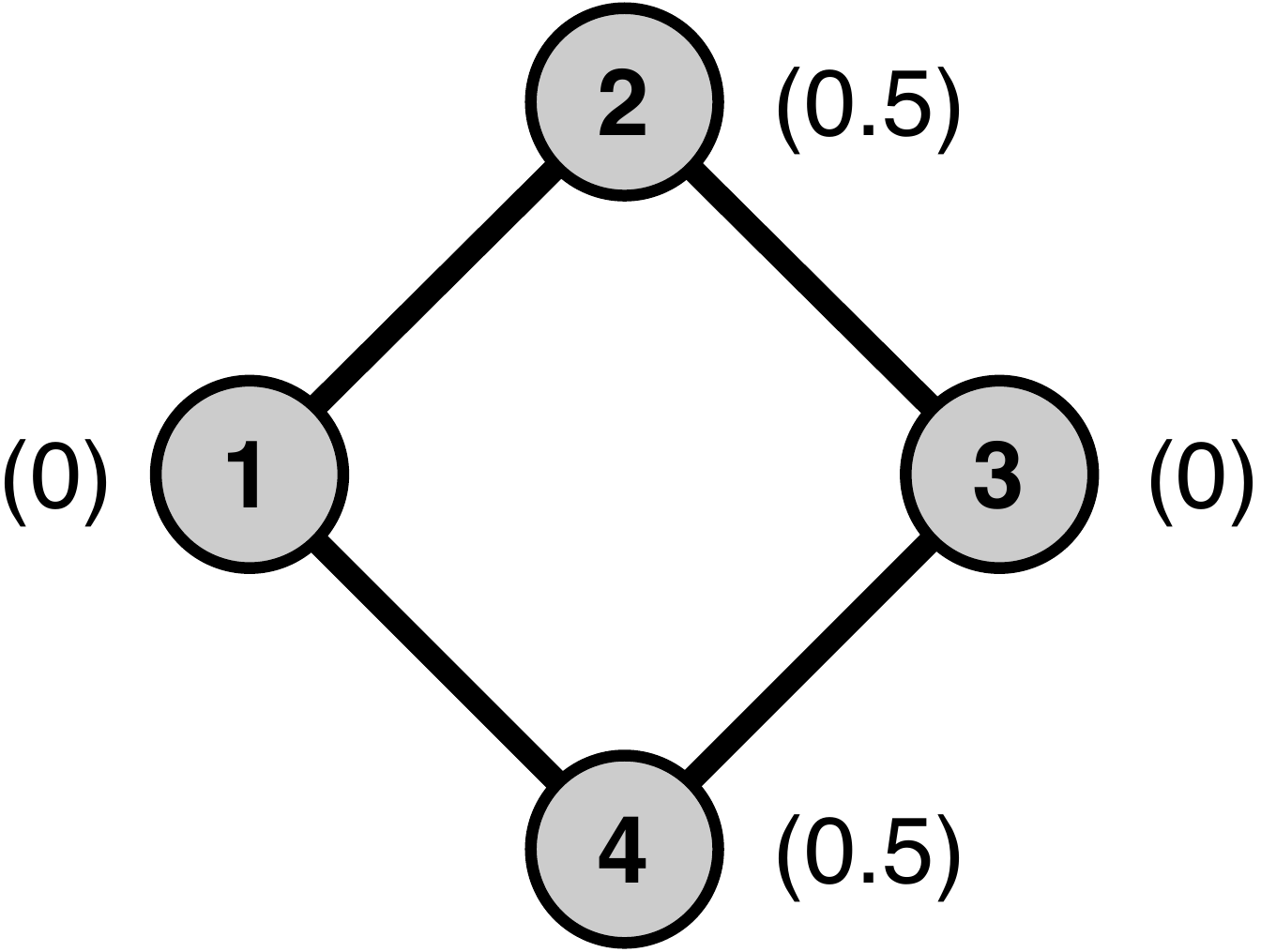}\\
		(a)  & (b) \\ \\
		\includegraphics[scale=0.27]{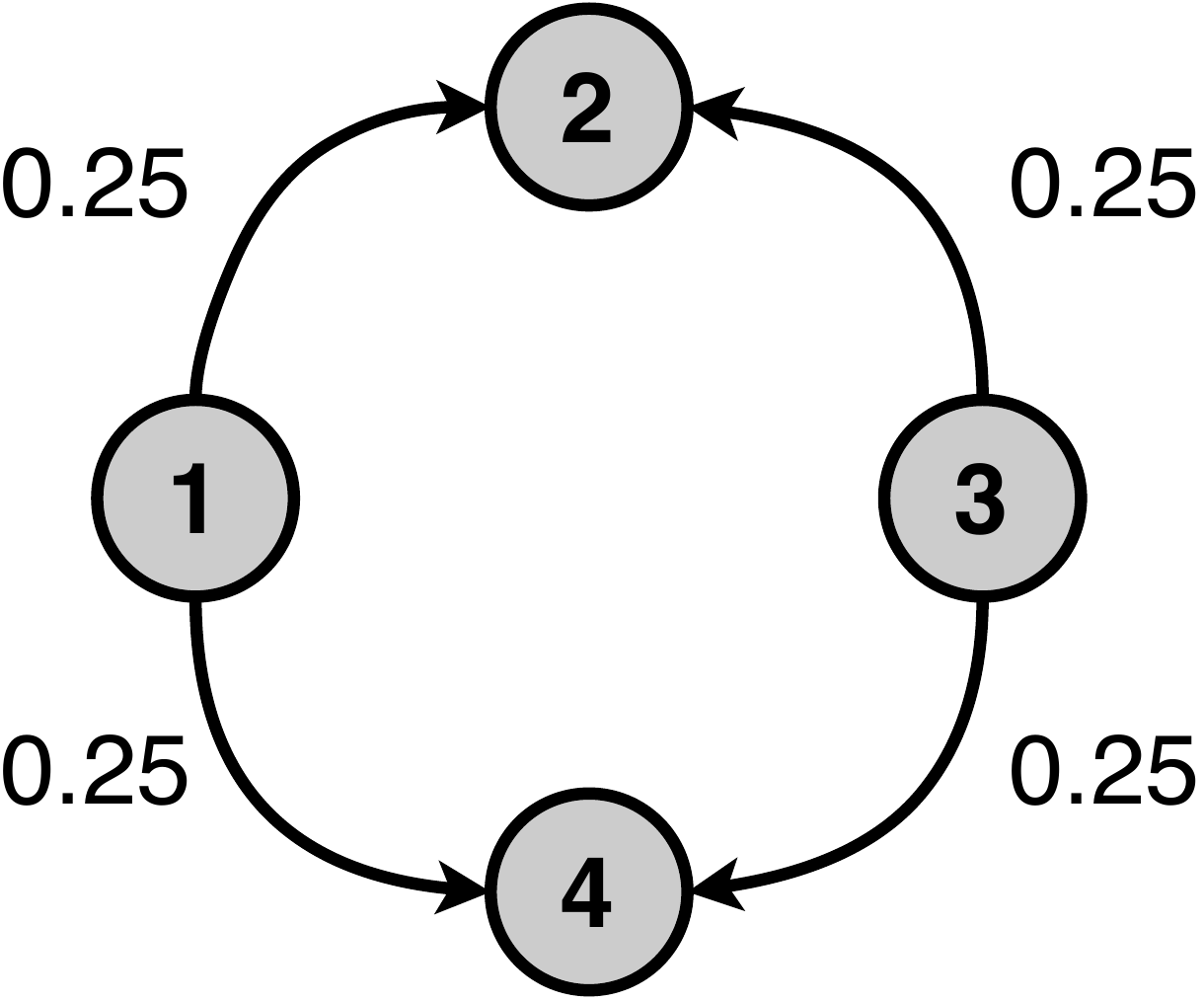} & \includegraphics[scale=0.27]{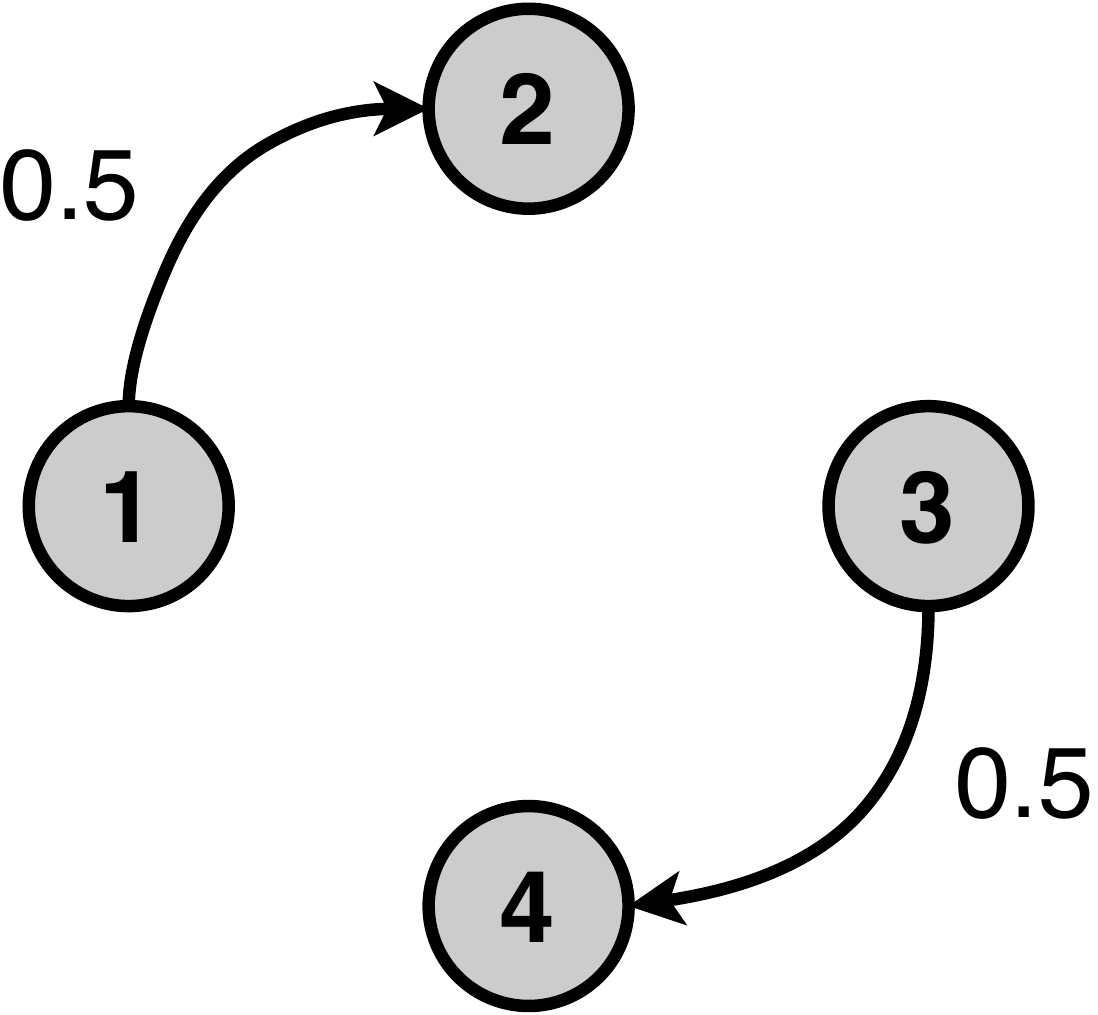}\\
		(c) & (d) \\
	\end{tabular}
	\end{center}
	\caption{Non-unique optimal outflow choice under NRPM. The
          number in (.) next to a node $i$ represents the population
          fraction $\x_i$. The number near arc $(i,j)$ represents
          $\d^*_{ij}$. (a) Initial population state (b) Optimal
          redistributed population state (c),(d) Different optimal
          outflow choices.}
	\label{fig:nrpm_opt_diff}
\end{figure}

However, for all the optimizers, the resultant node fractions
$\z^*_i$, for each node $i \in \node$ is unique. This property is a
direct application of the following lemma, whose proof is in
Appendix~\ref{app:aux-results}.

\begin{lemma}\thmtitle{Uniqueness of a subset of optimizer variables
    in a class of convex optimization
    problems} \label{lem:partial_strict_conv_prob}
  Let $\bld{v} \in \real^{n}$, $\bld{w} \in \real^{m}$. Consider the optimization problem
  \begin{equation*}
    \max_{(\bld{v},\bld{w}) \in \Omega} f(\bld{v}) .
  \end{equation*}
  Suppose that the set $\Omega \subset \real^{n + m}$ is non-empty and
  convex whereas the function $f(.)$ is a strictly concave function of
  $\bld{v}$. Then, every optimizer $(\bld{v}^i, \bld{w}^i)$ has the
  property that $\bld{v}^i = \bld{v}^*$, a unique constant. \bulletend
\end{lemma}

Direct application of Lemma~\ref{lem:partial_strict_conv_prob} to
problem $\prob_3$ gives us the following result.

\begin{lemma}\thmtitle{The resultant node fractions in any
    optimal solution of $\prob_3$ are unique}
  Consider the problem $\prob_3$ in \eqref{eq:prob_nrpm} with
  $\x \in \simplex$. Let $\mathcal{OP}$ be the set of all optimizers
  of $\prob_3$. Then $\forall \, (\z^*,\d^*) \in \mathcal{OP}$, $\z^*$
  is unique. \bulletend
  \label{lem:p3_unique}
\end{lemma}

We can make some more observations about the optimizers of $\prob_3$,
which we state in the following lemma. Its proof is in
Appendix~\ref{app:NRPM-proofs}.
\begin{lemma}\thmtitle{Properties of optimizers of
    $\prob_3$} \label{lem:p3_opt}
  Suppose $(\z^*,\d^*)$ is an optimizer of $\prob_3$. Then
  $\forall \, i \in \node$ there exists a constant $\overline{\eta}_i$
  such that
  \begin{equation}
    \label{eq:nrpm_level_same_neigh}%
    \lev_j(\z^*_j) = \overline{\eta}_i, \quad \forall j \in
    \bneigh^{\,i} \,\,\text{s.t.}\,\, \d^*_{ij} > 0 \,.
  \end{equation}

Also,
\begin{equation}
	\lev_j(\z^*_j) \geq \lev_i(\z^*_i), \,\, \forall j \in \neigh^i \,\,\text{s.t.}\,\, \d^*_{ij} > 0 \,.
	\label{eq:level_order_nrpm}
\end{equation}
\bulletend
\end{lemma}

\begin{remark}\thmtitle{Network restricted payoff maximization
    dynamics} %
  Problem $\prob_3$ gives the optimizer $\z^*$ of the social utility,
  $U(.)$, under the assumption that the agents may revise their
  choices to one of the nodes neighboring their current node, though
  with complete knowledge of the population state $\x$. Note that
  $\prob_3$ is feasible for each $\x \in \simplex$ as
  $\z_i = \d_{ii} = \x_i$, $\forall i \in \node$ and $\d_{ij} = 0$,
  $\forall (i,j) \in \arc$ is a feasible solution. Then, we let the
  evolution of $\x$ as a whole be~\eqref{eq:gen_dyn} with
  $\delta_{ij} = \d^*_{ij}$ for all $(i,j) \in \arc$, where
  $\{\d^*_{ij}\}_{(i,j) \in \arc}$ comes from an arbitrary optimizer
  of $\prob_3$. Note that, defined this way, even though we have
  non-uniqueness of $\delta_{ij}$'s, the effect is unique as in
  Lemma~\ref{lem:p3_unique}, we have established that
  $\z^*(\x) = \x + \inmat\, \Delta(\x)$ is uniquely defined for each
  $\x \in \simplex$. The resulting dynamics is
  \begin{equation}
    \label{eq:NRPM_ODE}
    \dot{\x} = \inmat \,\Delta(\x) = \z^*(\x) - \x .
  \end{equation}
  We refer to the dynamics~\eqref{eq:NRPM_ODE} as \emph{network
    restricted payoff maximization} (NRPM).
  \bulletend
\end{remark}

Note that~\eqref{eq:level_order_nrpm} is not the same as strong
positive correlation of $\inmat \Delta(\x)$ with $\lev(.)$ as in
\eqref{eq:level_order_nrpm}, the condition on $\lev(.)$ involves $\z^*$
instead of $\x$.

\begin{example}\thmtitle{NRPM may not satisfy strong positive correlation with $\lev(.)$}
Consider the graph with node set $\node = \{1,2,3\}$ and edge set $\edg = \{\{1,2\},\{2,3\}\}$. Let the cumulative payoff functions be of the form $\pe_i(y) = -0.5 \, y^2 - a_i \, y$ and hence the payoff density functions are of the form $\lev_i(y) = -y - a_i$. Let $a_1 = a_2 = 2$ and $a_3 = 0$ and let $\x = [0.2,0.8,0]^T$. For such a setup there is a unique optimizer $\d^*$ of $\prob_3$. A quick computation will reveal that $\d^*_{12} = 0.2$ even though $\lev_1(0.2) > \lev_2(0.8)$.
\bulletend
\end{example}

\subsection{Existence and Uniqueness of Solutions for NRPM}

Similar to NBRD, as $\dyn(.)$ for NRPM is obtained through an
optimization problem, analyzing the dynamics directly from its
definition can be cumbersome. In this subsection, we use Lemma
\ref{lem:p3_opt} to analyze Lipschitzness of $\z^*(\x)$. The idea is
to appropriately partition the set of nodes into components on which
$\lev_i(\z^*_i)$ is the same and then determining $\z^*$
using~\eqref{eq:nrpm_level_same_neigh}. This further leads to a
covering of the simplex $\simplex$, where in each set of the covering,
$\z^*(\x)$ is locally Lipschitz so that we can again apply
Lemma~\ref{lem:gen_dyn_EU}. Though this idea is similar to the one we
used for NBRD, here the situation is complicated by the fact that
there is coordination among the agents across the network.  This leads
to the construction being dependent on the network as a whole and not
just the immediate neighborhood of each node, as in NBRD.

\subsubsection*{Construction of the origin-destination graph and the
  origin-destination sets}

First consider the set $\barc$ and its power set
$\Mcal{P}(\barc) \ldef \{ \M \ | \ \M \subseteq \barc \}$. We can
compute candidate values of $\z^*$ assuming the support for
$\nu^*_{ij}$'s is $\M$, for each $\M \subseteq \barc$. Then, we can
determine the set of $\x$ for which the support of $\nu^*_{ij}$ is
indeed $\M$. For this purpose, we first consider the graph
$(\node, \M)$ and then construct a new graph $(\hat{\node}, \hat{\M})$
and pairs of subsets of nodes $(\Mcal{O}^p, \Mcal{D}^p)$ as follows.

\begin{itemize}
\item For each node $i \in \node$, we construct two nodes in
  $\hat{\node}$, labeled $i_o$ and $i_d$, standing for ``origin'' and
  ``destination''. For each arc $(i,j) \in \M$, we construct an arc
  $(i_o, j_d) \in \hat{\M}$. If an arc $(i, i) \in \M$ then the
  corresponding arc in $\hat{\M}$ is $(i_o, i_d)$. The graph
  $(\hat{\node}, \hat{\M})$ contains no other nodes or
  arcs.
\item For each of the weakly connected components
  $(\hat{\node}_r, \hat{\M}_r)$ of the graph $(\hat{\node}, \hat{\M})$
  that is non-trivial (having more than one node), we define origin
  and destination sets, $\Mcal{O}^r$ and $\Mcal{D}^r$, as
  \begin{equation}
    \Mcal{O}^r \ldef \{ i \in \node \, | \, i_o \in \hat{\node}_r \},
    \quad \Mcal{D}^r \ldef \{ i \in \node \, | \, i_d \in
    \hat{\node}_r \} . \label{eq:Op-Dp}
  \end{equation}
\end{itemize}

\begin{figure}
  \begin{center}
  \begin{tabular}{cc}
    \includegraphics[scale=0.3]{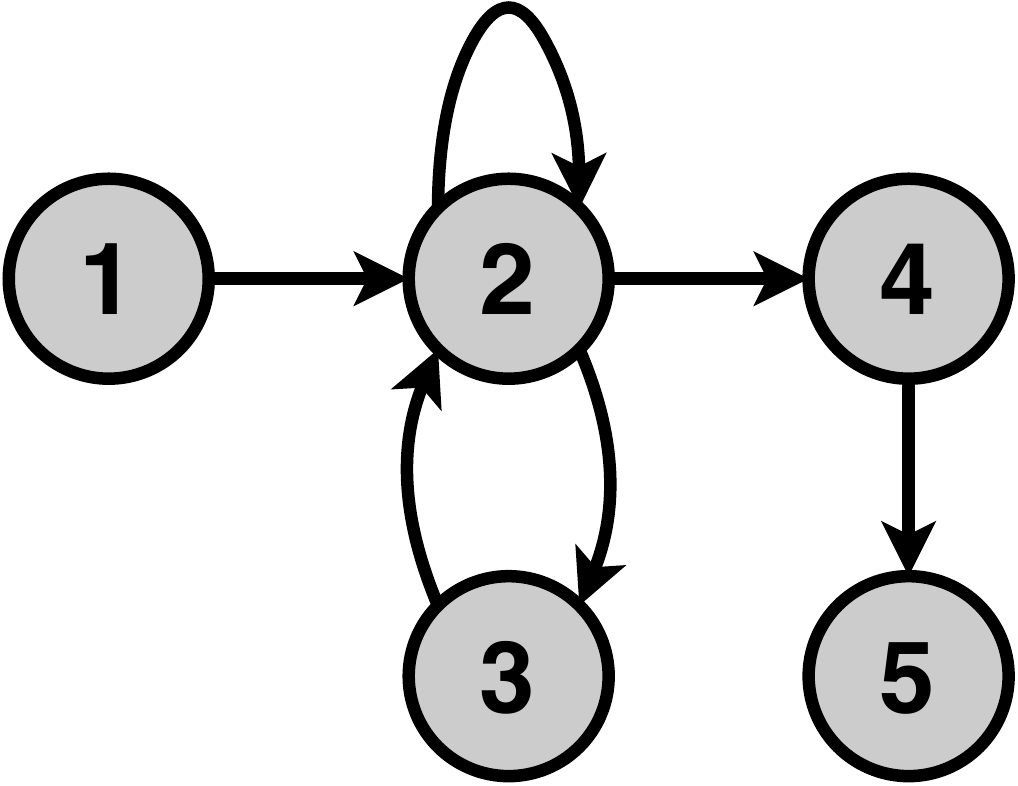} & \includegraphics[scale=0.3]{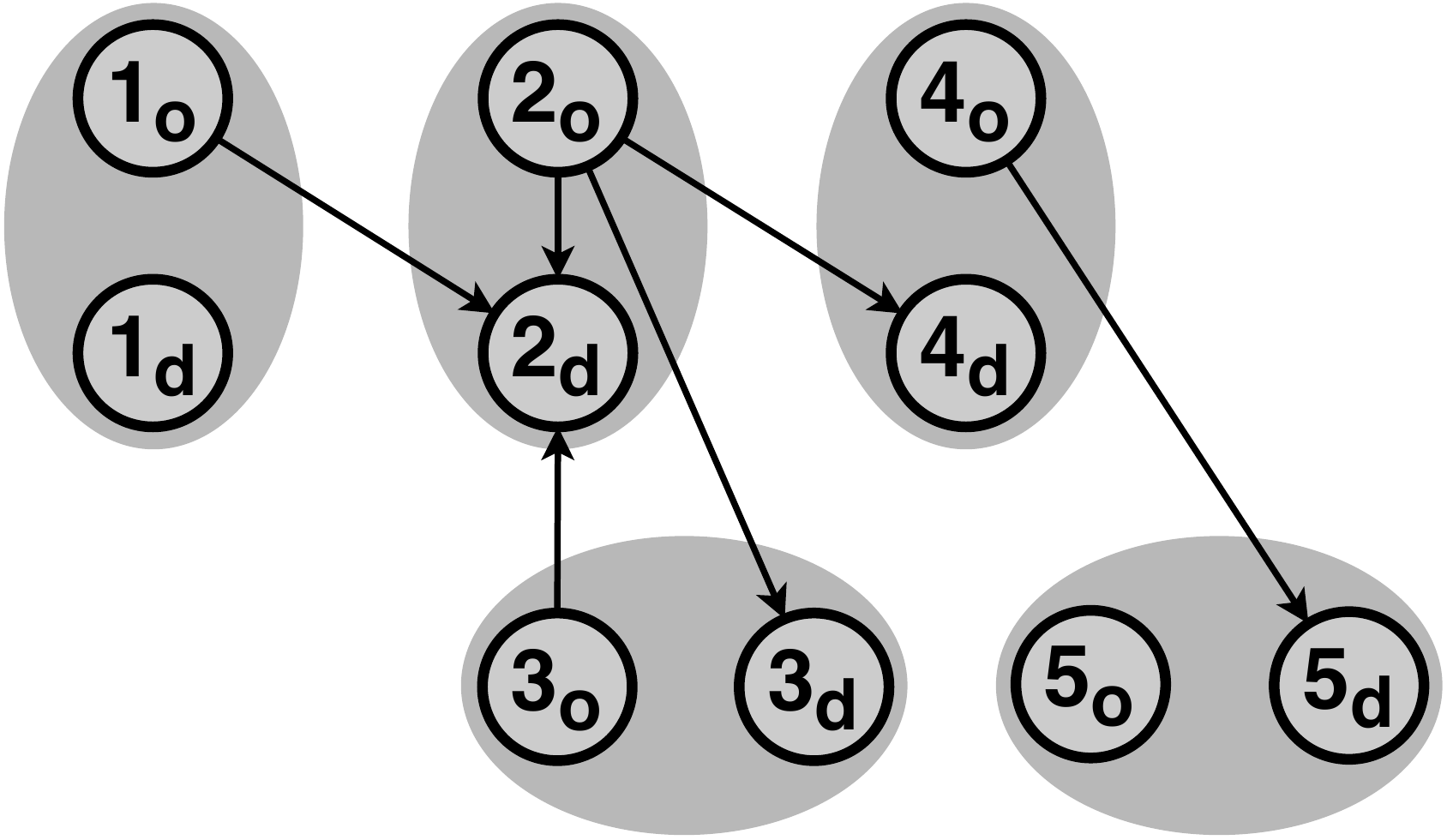}\\
    (a)  & (b)
  \end{tabular}
  \end{center}
  \caption{Construction of origin and destination sets. (a) The
    directed graph $(\node, \M)$. (b) The constructed graph
    $(\hat{\node}, \hat{\M})$. The non-trivial weakly connected
    components are $\{1_o,2_o,3_o,2_d,3_d,4_d\}$ and
    $\{4_o,5_d\}$. Thus, the two sets of origin and destination nodes
    are $\Mcal{O}^1 = \{1,2,3\}$, $\Mcal{D}^1 = \{2,3,4\}$ and
    $\Mcal{O}^2 = \{4\}$, $\Mcal{D}^2 = \{5\}$.}
  \label{fig:origin_dest}
\end{figure}

Figure~\ref{fig:origin_dest} shows an example of this
construction. This construction is useful in determining $\z^*$ if the
support of $\{\d^*_{ij}\}_{(i,j) \in \barc}$ is $\M$. More precisely,
we impose the conditions
\begin{equation}
  \nu_{ij} = 0, \ \forall (i,j) \in \M, \quad \d_{ij} = 0, \ \forall
  (i,j) \notin \M . \label{eq:nu-support-M}
\end{equation}
Finally, given $\M$ and the resulting $\Mcal{O}^r$'s and
$\Mcal{D}^r$'s, we define $F^\M(\x) : \real^N_+ \rightarrow \real^N_+$
as
\begin{equation}
  \label{eq:FM}
  F_i^\M(\x) =
  \begin{cases}
    \displaystyle 0, & i \notin \bigcup_{p \in [1,n]_\integer}
    \Mcal{D}^r\\
    \displaystyle \lev^{-1}_i \left( g^{-1}_{\Mcal{D}^r}\left(\sum_{i
          \in \Mcal{O}^r} \x_i \right)\right), & i \in \Mcal{D}^r ,
  \end{cases}
\end{equation}
with $g_{\Mcal{D}^r}(.)$ defined in \eqref{eq:g_func_brd}. Here
$n \leq N$ is the number of non-trivial weakly connected components in
$(\hat{\node},\hat{\M})$. From the proof of
Theorem~\ref{th:NBRD_full}, we know that $g^{-1}_{\Mcal{D}^r}(.)$ is
continuously differentiable. Hence, we can say that $F^\M(\x)$ is also
a continuously differentiable function of $\x$. We are now ready to
present the connection between this construction and the optimization
problem $\prob_3$. We provide the proof of the following result in
Appendix~\ref{app:NRPM-proofs}.

\begin{lemma}\thmtitle{Optimizers of $\prob_3$ with support pattern
    $\M$}\label{lem:p3-support-pattern}
  Let $(\z^*,\d^*)$ be an optimizer of $\prob_3$ with the support
  pattern~\eqref{eq:nu-support-M} for some $\M \subseteq \barc$. Let
  $\{\Mcal{O}^r\}_{p \in [1,n]_\integer}$ and
  $\{\Mcal{D}^r\}_{p \in [1,n]_\integer}$ be defined as
  in~\eqref{eq:Op-Dp} with $n$ being the number of non-trivial weakly
  connected components in the graph $(\hat{\node}, \hat{\M})$. Then,
  the following are true.
  \begin{itemize}
  \item[(a)] $\exists (i,j) \in \barc \cap \M$, for all $i \in \node$
    such that $\x_i > 0$.
  \item[(b)] Support of $\x$ is a subset of
    $\displaystyle \cup_{r \in [1,n]_\integer} \Mcal{O}^r$.
  \item[(c)] $\z^*(\x) = F^\M(\x)$, where $F^\M(.)$ is given
    in~\eqref{eq:FM}.  \bulletend
  \end{itemize}
\end{lemma}

Thus, an alternate way to solve $\prob_3$ would be to start off with a
candidate $\M \subseteq \barc$, construct the $\Mcal{O}^r$'s and
$\Mcal{D}^r$'s accordingly, compute $(\z, \d)$ and among all such
$(\z, \d)$ that are feasible for $\prob_3$, we can pick the ones that
maximize the objective function of $\prob_3$. While this procedure is
long-winded and may not be useful if our only interest was to solve
$\prob_3$, it helps us in proving Lipschitzness of $\z^*(.)$ using
Lemma~\ref{lem:gen_dyn_EU}. We summarize this in the following lemma,
whose proof is in Appendix~\ref{app:NRPM-proofs}.
 
\begin{lemma}\thmtitle{Existence and uniqueness of solutions for
    NRPM}%
  The state equation in \eqref{eq:NRPM_ODE} with an initial condition
  $\x(0) \in \simplex$ has a unique solution $\forall \, t \geq 0$.
  \label{lem:NRPM_UE}
  \bulletend
\end{lemma}

\subsection{On the Convergence of NRPM}

Here, we first characterize the set of equilibrium points of NRPM and
then demonstrate that the dynamics~\eqref{eq:NRPM_ODE} converges to
the equilibrium set asymptotically. The proof of the following result
appears in Appendix~\ref{app:NRPM-proofs}.

\begin{lemma}\thmtitle{Equilibrium set of NRPM}\label{lem:nrpm_eqpt} %
  For NRPM, the dynamics in \eqref{eq:NRPM_ODE}, $\eqpt$ the set of
  equilibrium points in $\simplex$ is the set $\ne$
  in~\eqref{eq:Nash_eq}. \bulletend
\end{lemma}

We can now show that the trajectories of NRPM in~\eqref{eq:NRPM_ODE}
asymptotically converges to the set $\ne$.

\begin{theorem}\thmtitle{NRPM converges asymptotically to the Nash
    equilibrium set}
  Let the evolution of $\x$ be governed by~\eqref{eq:NRPM_ODE}. If
  $\x(0) \in \simplex$ then as $t \to \infty$, $U(\x(t))$ converges to
  a constant and $\x(t)$ approaches the set $\ne$ in
  \eqref{eq:Nash_eq}.
  \label{th:PM_conv}
\end{theorem}

\begin{proof}
  First notice that the simplex $\simplex$ is positively invariant
  under the dynamics~\eqref{eq:NRPM_ODE}. Recall that $\prob_3$ is
  convex in $(\z, \d)$ and the function $V(.) \ldef -U(.)$ is strictly
  convex. Also, note that $\z = \x$, $\d_{ij} = 0$,
  $\forall (i,j) \in \arc$, $\d_{ii} = \x_i$, $\forall i \in \node$ is
  a feasible solution for $\prob_3$ for all $\x \in
  \simplex$. Further, $\z^*(\x)$ is unique by Lemma
  \ref{lem:p3_unique} for all the optimizers
  $( \ \z^*(\x), \ \d^*(\x) \ )$ of $\prob_3$. Then,
  $U(\z^*(\x)) \geq U(\x)$ or equivalently,
  \begin{equation}
    V(\z^*(\x)) \leq V(\x), \quad \forall \x \in \simplex .
    \label{eq:nrpm_disc_time_proof1}
  \end{equation}
  Now, as $V(.)$ is a twice differentiable strictly convex function,
  \begin{equation*}
    \nabla V(\x).[\z^*(\x) - \x] + V(\x) \leq V(\z^*(\x)) ,
  \end{equation*}
  which implies
  \begin{equation*}
    \dot{V}(\x) = \nabla V(\x).\dot{\x} \leq  V(\z^*(\x)) - V(\x) \leq
    0, \quad \forall \x \in \simplex.
  \end{equation*}
  Here, we use the fact that in \eqref{eq:NRPM_ODE},
  $\dot{\x} = \z^*(\x) - \x$ and
  \eqref{eq:nrpm_disc_time_proof1}. Thus $V(\x(t))$ and hence $U(\x(t))$ converges to a constant. Further, $\dot{V}(\x) = 0$ implies
  $V(\x) \leq V(\z^*(\x))$, which along
  with~\eqref{eq:nrpm_disc_time_proof1} means $V(\z^*(\x)) =
  V(\x)$. But recall that for each $\x \in \simplex$, $\z(\x) = \x$ is
  feasible for $\prob_3$. Also, by Lemma~\ref{lem:p3_unique} we know
  that for all optimizers of $(\z^*(\x), \d^*(\x))$ of $\prob_3$,
  $\z^*(\x)$ is unique. Thus, it must be that $\z^*(\x) = \x$, that is
  $\x \in \eqpt = \ne$ (see Lemma~\ref{lem:nrpm_eqpt}). Then, by
  LaSalle's invariance principle \cite{HK:2002:nsc}, the dynamics in
  \eqref{eq:NRPM_ODE} converges to the set $\ne$. 
\end{proof}

\section{Conclusion} \label{sec:conclusion}

We proposed three dynamics that model the evolution of a stratified
population of optimum seeking agents under different levels of
coordination. For the case with selfish agents, we generalized the
standard Smith dynamics to our setting. For the case with nodal level
coordination, we proposed the dynamics based on the best response of
the fraction of population in a node. The third dynamics is a
centralized dynamics with population level coordination achieved
through network restricted gradient ascent of the social utility. For
all three dynamics, we showed existence and uniqueness of solutions
and also asymptotic convergence of the social utility to a constant
and convergence of the population state to the set of Nash
equilibria.

Future work includes extending this framework to incorporate nodal
capacity constraints and establishing connections with opinion
dynamics. We would also like to utilize the results to come up with
efficient algorithms to compute the converging social
utility. Further, we would like to allow for changes in the cumulative
functions at rate much slower than that at which the social utility
converges and explore control problems.

\bibliographystyle{IEEEtran}
\bibliography{myref}

\begin{IEEEbiography}[{\includegraphics[width=1in,height=1.25in,clip,keepaspectratio]{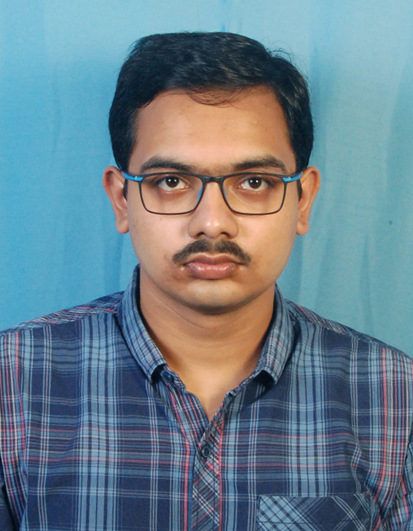}}]{Nirabhra Mandal} received the B.Tech. degree in Electrical Engineering from Institute of Engineering and Management, Salt Lake, Kolkata, India in 2017. Since 2018, he is pursuing the M.Tech(Res) degree from the Department of Electrical Engineering at the Indian Institute of Science. His research interests include multi-agent systems, population games, evolutionary dynamics on networks and non-linear control.
\end{IEEEbiography}
\vskip 0pt plus -1fil
\begin{IEEEbiography}[{\includegraphics[width=1in,
    height=1.25in,clip,keepaspectratio] {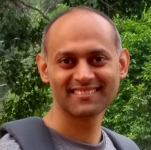}}]{Pavankumar
    Tallapragada} (S'12-M'14) received the B.E. degree in
  Instrumentation Engineering from SGGS Institute of Engineering $\&$
  Technology, Nanded, India in 2005, M.Sc. (Engg.) degree in
  Instrumentation from the Indian Institute of Science in 2007 and the
  Ph.D. degree in Mechanical Engineering from the University of
  Maryland, College Park in 2013. He was a Postdoctoral Scholar in the
  Department of Mechanical and Aerospace Engineering at the University
  of California, San Diego from 2014 to 2017. He is currently an
  Assistant Professor in the Department of Electrical Engineering and
  the Robert Bosch Centre for Cyber Physical Systems at the Indian
  Institute of Science. His research interests include networked
  control systems, distributed control, multi-agent systems and
  networked transportation systems.
\end{IEEEbiography} 

\appendix

\subsection{Proofs of Results on Strongly Positively Correlated
  Dynamics}\label{app:SPC-proofs}

\subsubsection{Proof of Lemma~\ref{lem:no_cycle}}

\begin{proofa}
  \emph{(By contradiction)} Suppose that $\exists \, n$ such that
  $\{1,\cdots,n\} \subseteq \node$ (possibly after relabeling nodes)
  and
  \begin{itemize}
  \item $\forall i \in [1,n-1]_\integer$, $i+1 \in \neigh^{i}$ and
    $1 \in \neigh^{n}$
  \item $\forall i \in [1,n-1]_\integer$, $\delta_{i (i+1)} > 0$ and
    $\delta_{n 1} > 0$.
  \end{itemize}
  This describes a directed cycle in $\tilde{\flow}$ with the nodes
  $(\, 1,\cdots,n,1 \, )$. Then the assumption that $\dyn(.)$ is
  strongly positively correlated implies
  \begin{equation*}
    u_{1}(\x_{1}) < \cdots < u_{n}(\x_{n}) < u_{1}(\x_{1}),
  \end{equation*}
  which is a contradiction.
\end{proofa}

\subsubsection{Proof of Lemma~\ref{lem:ne_fbd}}

\begin{proofa}
Consider the optimization problem
\begin{equation*}
  \begin{split}
    \prob_1: \ \max_{\x} U(\x) = \max_{\x} \sum_{i \in \node}
    \pe_i(\x_i), \, \text{ s.t. } \,\, \x \in \simplex \,.
  \end{split}
\end{equation*}
As $U(.)$ is a strictly concave function and $\simplex$ is a convex
set, $\prob_1$ has a unique optimizer $\x^* \in \simplex$. The
Lagrangian for $\prob_1$ is
\begin{equation*}
  \lag_1 = \sum_{i \in \node} \pe_i(\x_i) - \lambda \left(\sum_{i \in
      \node} \x_i - 1 \right) + \sum_{i \in \node} \mu_i \x_i,
\end{equation*}
where $\lambda \in \real$ and $\{\mu_i \geq 0\}_{i \in \node}$ are the
Lagrange multipliers. From the KKT conditions for $\prob_1$, we have
\begin{subequations}
\begin{align}
  \label{eq:kkt_grad} &\lev_i(\x_i^*) - \lambda + \mu_i = 0, \,
                        \forall i \in \node,\\
  \label{eq:kkt_comp_slack} &\mu_i \x_i^* = 0, \, \forall i \in \node \,.
\end{align}
\end{subequations}
Thus, we must have
\begin{equation}
  \mu^*_i = 0, \quad \lev_i(\x_i^*) - \lambda^* = 0, \quad \forall
  i \in \supp(\x^*).
  \label{eq:i_in_supp}
\end{equation}
Now, consider an arbitrary but fixed $i \in \supp(\x^*)$. Combining
\eqref{eq:kkt_grad} and the fact that $\mu^*_j \geq 0$, we get
\begin{equation}
  \lev_j(\x_j^*) - \lambda^* \leq 0, \quad \forall j \in \neigh^i .
  \label{eq:j_neigh_supp}
\end{equation} 
Combining \eqref{eq:i_in_supp} and \eqref{eq:j_neigh_supp} we get
\begin{equation*}
  \lev_j(\x_j^*) \leq \lev_i(\x_i^*), \quad \forall j \in \neigh^i, \
  \forall i \in \supp(\x^*) .
\end{equation*}
Thus, we see that $\x^* \in \ne$ and as a result, $\ne$ is non-empty.

Finally, note that if $\dyn(.)$ is strongly positively correlated with
$\lev(.)$ then the fact that $\ne$ is a subset of the equilibrium set
$\eqpt$ follows directly from Definition \ref{def:PC} and the
assumption that $\delta_{ij} = 0$ $\forall j \in \neigh^i$, if
$\x_i = 0$. This completes the proof.
\end{proofa}

\subsection{Proof of Lemma~\ref{lem:SSD-spc-eqpt}
  (SSD)} \label{app:SSD-lemma}

\begin{proofa}
  Consider arbitrary nodes $i,j \in \node$. If $j \notin \neigh^i$
  then by definition $\delta_{ij} = 0$. If $j \in \neigh^i$ and $\x$
  is such that $\lev_i(\x_i) \geq \lev_j(\x_j)$ then as $\lev_i(.)$'s
  are strictly decreasing, we have
  $\x_i \leq \lev_i^{-1}(\lev_j(\x_j))$. Then, \eqref{eq:yij} implies
  $y_{ij} = \x_i$ and \eqref{eq:sd_delta_int} further implies that
  $\delta_{ij} = 0$. Thus, SSD is strongly positively correlated with
  $\lev$.

  Now, using Lemma \ref{lem:no_cycle}, it is immediately evident that
  the only $\Delta(\x) \in \ker(\inmat) \cap \simplex$ is
  $\Delta(\x) = \bld{0}$. Thus $\x$ is an equilibrium point if and
  only if all $\delta_{ij}(\x)$ are \emph{zero}. Now, for a node
  $i \in \node$ and $j \in \neigh^i$, $\delta_{ij} = 0$ can occur in
  one of two ways: $i \notin \supp(\x)$ or $i \in \supp(\x)$. If
  $i \notin \supp(\x)$ then by \eqref{eq:sd_delta_int}
  $\delta_{ij} = 0$, $\forall j \in \neigh^i$ is the only
  possibility. On the other hand, if $i \in \supp(\x)$ then we have
  two sub-cases: $y_{ij} = \x_i$ or
  $\lev_i^{-1}(\lev_j(\x_j)) \leq y_{ij} < \x_i$ and yet the integral
  in~\eqref{eq:sd_delta_int} is zero. In the second sub-case,
  $\delta_{ij} = 0$ iff
\begin{align*}
  \lev_j(\x_j) ( \x_i - y_{ij} ) %
  &\leq \int_{y_{ij}}^{\x_i} \lev_i(y) \mathrm{d}y \\
  &< \lev_i(y_{ij}) (\x_i - y_{ij}) \leq \lev_j(\x_j) (\x_i - y_{ij}) ,
\end{align*}
where the inequalities are due to the strictly decreasing nature of
$\lev_i$. Now, it is clear that this inequality and as a result the
case $\x_i > \lev_i^{-1}(\lev_j(\x_j)) = y_{ij}$ cannot occur. Thus,
the case $\delta_{ij} = 0$ and $\x_i > 0$ occur iff $y_{ij} = \x_i$,
which according to~\eqref{eq:yij} occurs iff
$\lev_i(\x_i) \geq \lev_j(\x_j)$ for all $j \in \neigh^i$. Hence, the
equilibrium set $\eqpt$ of SSD, $\eqpt$, is $\ne$ in
\eqref{eq:Nash_eq}.
\end{proofa}

\subsection{Proofs of Results on NBRD} \label{app:NBRD-proofs}

\subsubsection{Proof of Lemma~\ref{lem:opt_prob2i}}

\begin{proofa}
  The cost function in~\eqref{eq:p_nbrd-prob} is strictly concave in
  $\{\d_{ij}\}_{j \in \bneigh^{\,i}}$ and the
  problem~\eqref{eq:p_nbrd-prob} is a convex program. Moreover, as
  $\x_i \geq 0$, the problem $\prob^{\,i}_2$ is always feasible since
  $\d_{ij} = 0$, $\forall j \in \neigh^i$ and $\d_{ii} = \x_i$,
  $\forall i \in \node$ is feasible. This justifies the uniqueness of
  optimizers \cite{SB-LV:2004:cvx}. The Lagrangian for $\prob^{\,i}_2$
  can be written as
\begin{equation*}
  \begin{split}
    \lag^i_2 &= \sum_{j \in \neigh^i}[\pe_j(\x_j+\d_{ij}) -
    \pe_j(\x_j)] + [\pe_i (\d_{ii}) - \pe_i(0)]\\
    & \quad - \lambda_i \left(\sum_{j \in \bneigh^{\,i}}
      \d_{ij} - \x_i \right) + \sum_{j \in \bneigh^{\,i}}
    \mu_{ij}\,\d_{ij}\, ,
	\end{split}
\end{equation*}
where $\lambda_i$ and $\{\mu_{ij} \geq 0\}_{j \in \bneigh^{\,i}}$ are
the Lagrange multipliers. From the KKT conditions for $\prob^{\,i}_2$,
we have
\begin{subequations}
  \label{eq:nbrd_kkt}
\begin{align}
	\label{eq:p_nbrd_kkt_grad1} &\lev_i(\d_{ii}) - \lambda_i +
                                \mu_{ii} = 0,\\
  \label{eq:p_nbrd_kkt_grad2} &\lev_j(\x_j + \d_{ij}) - \lambda_i + \mu_{ij} = 0, \, \forall j \in \neigh^{i},\\
	\label{eq:p_nbrd_kkt_comp_slack} &\mu_{ij} \d_{ij} = 0, \, \forall j \in \bneigh^{\,i} \,.
\end{align}
\end{subequations}

%

Now, $\forall j \in \bneigh^{\,i}$ such that $\d^*_{ij} > 0$, by
\eqref{eq:p_nbrd_kkt_comp_slack} we have $\mu^{*}_{ij} = 0$. Then we
obtain~\eqref{eq:level_same_own} and~\eqref{eq:level_same_neigh}
directly from~\eqref{eq:p_nbrd_kkt_grad1}
and~\eqref{eq:p_nbrd_kkt_grad2}, respectively by setting
$\eta_i = \lambda_i^*$. Now $\forall j \in \bneigh^{\,i}$ such that
$\d^*_{ij} = 0$, we have $\mu^{*}_{ij} \geq 0$. This along
with~\eqref{eq:p_nbrd_kkt_grad1} and~\eqref{eq:p_nbrd_kkt_grad2}
gives~\eqref{eq:level_diff} with $\eta_i = \lambda^*_i$.
\end{proofa}

\subsubsection{Proof of Lemma~\ref{lem:nbrd_str_pos_cor}}

\begin{proofa}
  Consider an arbitrary but fixed $i \in \node$. 
  Notice then from~\eqref{eq:level_same_own} and~\eqref{eq:level_diff} that
  $\lev_i(\d^*_{ii}) \leq \eta_i$. Now, if $\d^*_{ij} > 0$ for some $j \in \neigh^i$
  then $\d^*_{ii} < \x_i$ and as $\lev_j(.)$'s are strictly
  decreasing functions for all $j \in \node$, we have the following
\begin{equation*}
  \lev_j(\x_j) > \lev_j(\x_j + \d^*_{ij}) = \eta_i \geq
  \lev_i(\d^*_{ii}) > \lev_i(\x_i),
\end{equation*}
where the equality is due to~\eqref{eq:level_same_neigh}. Thus
$\d^*_{ij} > 0$ implies $\lev_j(\x_j) > \lev_i(\x_i)$. Applying
the contrapositive of this statement to all $(i,j) \in \arc$ completes
the proof.
\end{proofa}

\subsubsection{Proof of Lemma~\ref{lem:nbrd_alt_dyn}}

\begin{proofa}
  We can immediately make the following observations from the meaning
  of support and Lemma~\ref{lem:opt_prob2i}.
  \begin{itemize}
  \item As $\M^i$ is the support of the optimizers of $\prob^{\,i}_2$,
    $\d^*_{ik} = 0$, for all $k \notin \M^i$.
  \item If $i \in \M^i$ then $\d^*_{ii} = \lev^{-1}_i(\eta_i)$,
    otherwise $\d^*_{ii} = 0$.
  \item $\d^*_{ij} = \lev^{-1}_j(\eta_i) - \x_j$,
    $\forall j \in \Mcal{M}^i \setminus \{i\}$.
  \end{itemize}
  Then, by the feasibility conditions of $\prob^{\,i}_2$ we have
  \begin{equation*}
    \sum_{j \in \Mcal{M}^i} \d^*_{ij} = \x_i
  \end{equation*}
  which implies
  \begin{equation}
    \sum_{j \in \Mcal{M}^i} \lev^{-1}_j(\eta_i) = \x_i + \sum_{j \in
      \Mcal{M}^i \setminus \{i\}} \x_j \,.
    \label{eq:find_etai}
  \end{equation}  
  Then \eqref{eq:eta_brd} and \eqref{eq:delta_ij_brd} follow
  immediately from~\eqref{eq:find_etai} and the definition of the
  function $g_{\M^i}(.)$ in \eqref{eq:g_func_brd}. Similarly,
  \eqref{eq:delta_ii_brd} follows from the fact that if $i \in \M^i$
  then $\d^*_{ii} = \lev^{-1}_i(\eta_i)$. Note that, as $\lev_i(.)$ is
  strictly decreasing and continuous for each $i$, so are
  $\lev^{-1}_i(.)$ and $g_{\M^i}(.)$. Thus the inverse function
  $g^{-1}_{\M^i}(.)$ exists.
\end{proofa}

\subsubsection{Proof of Theorem~\ref{th:NBRD_full}}

\begin{proofa}
  To show existence and uniqueness of solutions, we consider
  $\real^N_+$ to be the domain of $\Delta(\x)$, instead of restricting
  $\x$ to be in $\simplex$. Consider the problem $\prob^{\,i}_2$ for
  an arbitrary $i$. Suppose the support of
  $\{\d^*_{ij}\}_{j \in \bneigh^{\,i}}$ is known a priori. Then, the
  optimizer of $\prob^{\,i}_2$ can be found easily using
  \eqref{eq:delta_brd}. Alternatively, we can solve the problem
  $\prob^{\,i}_2$ by computing the candidate
  $\{\d_{ij}\}_{j \in \bneigh^{\,i}}$ using \eqref{eq:delta_brd} for
  each candidate support $\M^i_s \subseteq \bneigh^{\,i}$ and picking
  the one that satisfies feasibility and maximizes the objective
  function of the problem $\prob^{\,i}_2$.

  In particular, given a candidate support
  $\M^i_s \subseteq \bneigh^{\,i}$, we let $\eta_i(\M^i_s)$ be as
  in~\eqref{eq:eta_brd} with $\M^i$ replaced by $\M^i_s$. Then,
  letting $w_j \ldef \lev^{-1}_j(\eta_i(\M^i_s))$, the feasibility
  conditions of $\prob^{\,i}_2$ can be re-written as
  \begin{subequations}
    \label{eq:eta_feasibility}
    \begin{align}      
      &w_j - \x_j \geq 0, \quad \forall j \in
        \M^i_s \setminus \{ i \} \\
      &w_i \geq 0, \quad \text{if } i \in \M^i_s \\
      &\sum_{j \in \Mcal{M}^i_s} w_j = \x_i +
        \sum_{j \in \Mcal{M}^i_s \setminus \{i\}} \x_j \\
      &w_j \geq \x_j, \quad \forall j \in
        \bneigh^{\, i} \setminus \M^i_s .
    \end{align}
  \end{subequations}
  Note that the last condition is a re-expression
  of~\eqref{eq:level_diff}. We can also write the objective function
  of $\prob^{\,i}_2$ directly in terms of $\eta_i({\M^i_s})$ as
  \begin{equation}
    Q_{\M^i_s}(\x) \ldef %
    \sum_{j \in \M^i_s \setminus \{i\}} \left[\pe_j(w_j) -
      \pe_j(\x_j) \right] + l \, \Big[\pe_i(w_i) - \pe_i(0) \Big]\,,
    \label{eq:Q_eta_x}
  \end{equation}
  where 
  \begin{equation*}
    l =
    \begin{cases}
      1, & \text{if } i \in \M^i_s\\
      0, & \text{otherwise}\,.
    \end{cases}
  \end{equation*}
  Thus, in the set
  \begin{align*}
    & \Mcal{C}_{\M^i_r} \ldef \bigg\{ \x \in \real^N_+ \,\,\Big|\,\,
      Q_{\M^i_r}(\x) \geq Q_{\M^i_q}(\x), \ \forall \M^i_q \subseteq
      \bneigh^{\,i} \\
    & \qquad \qquad \quad \text{ s.t. } \eqref{eq:eta_feasibility}
      \text{ satisfied with } \M^i_s \text{ replaced by } \M^i_q
      \bigg\} \,,
  \end{align*}
  the choice of the support of the optimizer
  $\{\d^*_{ij}\}_{j \in \bneigh^{\,i}}$ is $\M^i = \M^i_r$ as the
  objective function of $\prob^{\,i}_2$ is maximized and the
  feasibility constraints of $\prob^{\,i}_2$ are satisfied. Now note
  that in this set $\Mcal{C}_{\M^i_r}$,
  $\{\d^*_{ij}\}_{j \in \bneigh^{\,i}}$ is given by
  \eqref{eq:delta_ij_brd} with $\M^i = \M^i_r$. Since the payoff
  density functions $\lev_j(.)$'s are strictly decreasing and
  continuously differentiable, their derivatives are never zero. Hence
  by the inverse function theorem, their inverse functions are also
  continuously differentiable. This same reasoning can be done for the
  functions $g_{\M^i}(.)$ and $g^{-1}_{\M^i}(.)$. Thus $\delta_{ij}$
  is continuously differentiable and hence locally Lipschitz in each
  $\Mcal{C}_{\M^i_r}$. Now as the set $\bneigh^{\,i}$ is finite, so is
  the set of all its subsets. Thus, the number of sets
  $\Mcal{C}_{\M^i_s}$ is also finite. Moreover, from their definition,
  these sets are also closed. Such coverings can be found
  $\forall i \in \node$. Thus, by Lemma \ref{lem:gen_dyn_EU}, for each
  $\x(0) \in \simplex$, we have existence and uniqueness of solutions
  of NBRD $\forall t \geq 0$.

  Now, by Lemma \ref{lem:nbrd_str_pos_cor} we know that NBRD is
  strongly positively correlated with $\lev(.)$. Then, by Lemma
  \ref{lem:no_cycle} it is evident that the only
  $\Delta(\x) \in \ker(\inmat)$ is $\Delta(\x) = \bld{0}$. Now if
  $\Delta(\x) = 0$ then $\forall i \in \supp(\x)$,
  $\d^*_{ii} = \x_i > 0$. Then by combining
  \eqref{eq:p_nbrd_kkt_grad1} with \eqref{eq:p_nbrd_kkt_comp_slack} we
  get $\lev_i(\x_i) = \lambda^*_i$. Combining this with
  \eqref{eq:p_nbrd_kkt_grad2} and the fact that $\mu^*_{ij} \geq 0$ we
  get
  \begin{equation*}
    \lev_i(\x_i) \geq \lev_j(\x_j), \ \forall j \in \neigh^i, \
    \forall i \in \supp(\x) .
  \end{equation*}
  Thus, we can say that $\ne \supseteq \eqpt$, the set of equilibrium
  points in $\simplex$. Further, from Lemma~\ref{lem:ne_fbd}, we also
  know that $\ne \subseteq \eqpt$ and that it is non-empty. Thus, for
  NBRD $\eqpt = \ne$. Finally, application of Lemmas
  \ref{lem:nbrd_str_pos_cor} and Theorem~\ref{th:dyn_conv} completes
  the proof.
\end{proofa}

\subsection{Proofs of Results on NRPM} \label{app:NRPM-proofs}

\subsubsection{Proof of Lemma~\ref{lem:p3_opt}}

\begin{proofa}
  The lagrangian for $\prob_3$ can be written as
  \begin{equation*}
    \begin{split}
      \lag_3 &= \sum_{i \in \node}\pe_i(\z_i) - \sum_{i \in \node}
      \lambda_i \left(\z_i - \x_i - \sum_{j \in \bneigh^{\,i}}
        (\d_{ji} - \d_{ij})\right)\\
      & \quad - \mu_i \left( \sum_{j \in \bneigh^{\, i}}\d_{ij} - \x_i
      \right) + \sum_{(i,j) \in \barc} \nu_{ij}\,\d_{ij}\, ,
    \end{split}
  \end{equation*}
  where $\{\lambda_i\}_{i \in \node}$, $\{\mu_i\}_{i \in \node}$ and
  $\{\nu_{ij} \geq 0\}_{(i,j) \in \barc}$ are the Lagrange
  multipliers. Then, from the KKT conditions for $\prob_3$, we have
  \begin{subequations}
    \begin{align}
      \label{eq:p_nrpm_kkt_grad1}%
      &\lev_i(\z_i) - \lambda_i = 0, \, \forall i \in \node,\\
      \label{eq:p_nrpm_kkt_grad2}%
      &- \lambda_i + \lambda_j - \mu_i + \nu_{ij} = 0, \, \forall
        (i,j) \in \barc,\\
      \label{eq:p_nrpm_kkt_comp_slack2}%
      &\nu_{ij} \, \d_{ij} = 0, \, \forall (i,j) \in \barc \,.
\end{align}
\end{subequations}
Consider an arbitrary but fixed $i \in \node$. If $\d^*_{ij} > 0$ for
only one $j \in \bneigh^{\,i}$ then \eqref{eq:nrpm_level_same_neigh}
is trivially true. So now, suppose that $\d^*_{ij} > 0$ and
$\d^*_{ik} > 0$ for $j \neq k$ and $j, k \in \bneigh^{\,i}$. Then by
\eqref{eq:p_nrpm_kkt_comp_slack2}, $\nu^*_{ij} = \nu^*_{ik} = 0$
which, by \eqref{eq:p_nrpm_kkt_grad1} and \eqref{eq:p_nrpm_kkt_grad2},
gives
\begin{equation}
  \lev_j(\z^*_j) = \lev_k(\z^*_k) = \lambda_j = \lambda_k = \lambda_i +
  \mu_i .
  \label{eq:nrpm_same_lev_proof}
\end{equation}
This proves \eqref{eq:nrpm_level_same_neigh}.

Now, if $\d^*_{ii} > 0$ and $\d^*_{ij} > 0$ for some
$j \in \bneigh^{\,i}$ then by~\eqref{eq:nrpm_same_lev_proof}, we see
that~\eqref{eq:level_order_nrpm} is true. So, now suppose that
$\d^*_{ii} = 0$ and $\d^*_{ij} > 0$ for some $j \in
\bneigh^{\,i}$. Then, \eqref{eq:p_nrpm_kkt_grad2}
and~\eqref{eq:p_nrpm_kkt_comp_slack2} imply that
$\mu_i = \nu_{ii} \geq 0$. Now, using \eqref{eq:p_nrpm_kkt_grad2}
and~\eqref{eq:p_nrpm_kkt_comp_slack2} for $(i,j)$, we have
$\lambda_j = \lambda_i + \mu_i \geq \lambda_i$. Thus, we now
obtain~\eqref{eq:level_order_nrpm} by
using~\eqref{eq:p_nrpm_kkt_grad1} for $i$ and $j$.
\end{proofa}

\subsubsection{Proof of Lemma~\ref{lem:p3-support-pattern}}

\begin{proofa}
  \textbf{(a):} This claim follows immediately from the feasibility
  constraints~\eqref{eq:nrpm-xi-reallocation}
  and~\eqref{eq:nrpm-non-neg} in $\prob_3$.

  \textbf{(b):} This claim now follows from \textbf{(a)} and the
  construction of the $\Mcal{O}^r$ sets.

  \textbf{(c):} Notice from~\eqref{eq:Op-Dp} that if
  $i \notin \displaystyle \cup_{r \in [1,n]_\integer} \Mcal{D}^r$ then
  $\nexists (k,i) \in \M$. Then, from~\eqref{eq:nu-support-M}, we see
  that $\d^*_{ki} = 0$ for all $(k,i) \in \barc$, which along
  with~\eqref{eq:nrpm-flow-balance}-~\eqref{eq:nrpm-non-neg}, implies
  that $\z^*_i = 0$. Thus, for all
  $i \notin \displaystyle \cup_{p \in [1,n]_\integer} \Mcal{D}^r$, we
  have $\z^*_i = F^\M_i(\x)$.
  
  Now, from the construction of $\Mcal{O}^r$ and $\Mcal{D}^r$ and the
  support pattern~\eqref{eq:nu-support-M}, it is evident that
  \begin{equation}
    \sum_{i \in \Mcal{D}^r} \z^*_i = \sum_{i \in \Mcal{O}^r} \x_i , \
    \forall r \in [1,n]_\integer .
    \label{eq:orig-dest-balance}
  \end{equation}
  Further, applying Lemma~\ref{lem:p3_opt} with the support
  pattern~\eqref{eq:nu-support-M}, and the definition of $\Mcal{D}^r$,
  we see that
  \begin{equation}
    \lev_i(\z^*_i) = \overline{\eta}_{\Mcal{D}^r}, \ \forall i \in
    \Mcal{D}^r, \ \forall r \in [1,n]_\integer \, ,
    \label{eq:nrpm_same_level_set}
  \end{equation}
  for some $\overline{\eta}_{\Mcal{D}^r}$. Now,
  considering~\eqref{eq:orig-dest-balance}-\eqref{eq:nrpm_same_level_set}
  together gives us
  \begin{equation*}
    \overline{\eta}_{\Mcal{D}^r} = g^{-1}_{\Mcal{D}^r}\left(\sum_{i
        \in \Mcal{O}^r} \x_i \right), \forall r \in [1,n]_\integer,
  \end{equation*}
  from which claim~(c) now follows.
\end{proofa}

\subsubsection{Proof of Lemma~\ref{lem:NRPM_UE}}

\begin{proofa}
  Here, as in the proof of Theorem~\ref{th:NBRD_full}, instead of
  restricting $\x$ to $\simplex$, we consider $\real^N_+$ to be the
  domain of $\z^*(\x)$. For each $\M \subseteq \barc$, $F^\M(\x)$ is a
  candidate for $\z^*(\x)$. Of course, $F^\M(\x)$ may not even be a
  feasible value of $\z$ in $\prob_3$. Thus, as a function of $\x$, we
  let the \emph{set of feasible support patterns} be
  \begin{equation*}
    \Mcal{F}(\x) \ldef \{ \M \subseteq \barc \ | \ \exists \, \d \text{
      s.t. } \z = F^\M(\x), \
    \eqref{eq:nrpm-flow-balance}-\eqref{eq:nrpm-non-neg} \} .
  \end{equation*}
  For each $\x \in \real^N_+$, each optimizer $(\z^*, \d^*)$ of
  $\prob_3$, must satisfy the support pattern~\eqref{eq:nu-support-M}
  for some $\M \subseteq \barc$. Thus, for each $\x \in \real^N_+$,
  $\z^*(\x) = F^\M(\x)$ for some $\M \subseteq \barc$. Further, for a
  given $\M \subseteq \barc$, $\z^*(\x) = F^\M(\x)$ if
  $\x \in \Mcal{C}_\M$, where
  \begin{align*}
    \Mcal{C}_\M \ldef \{ %
    &\x \in \real^N_+ \ | \ \M \in \Mcal{F}(\x), \\
    &U(F^{\M}(\x)) \geq U(F^{\M'}(\x)), \ \forall \M' \in \Mcal{F}(\x) \}.
  \end{align*}

  Now, we seek to establish that if $\Mcal{C}_\M$ is non-empty then it
  is a closed subset of $\real^N_+$. Since $F^\M(.)$ is a continuously
  differentiable function for each $\M \subseteq \barc$ and since
  $U(.)$ is a strictly concave twice-continuously differentiable
  function, it suffices to show that
  \begin{align*}
    \Mcal{R}_\M %
    &\ldef \{ \x \in \real^N_+ \ | \ \M \in \Mcal{F}(\x)
      \} \\
    &= \{ \x \in \real^N_+ \ | \ \exists \, \d \text{
      s.t. } \z = F^\M(\x), \
      \eqref{eq:nrpm-flow-balance}-\eqref{eq:nrpm-non-neg} \}
  \end{align*}
  is closed. Notice that the constraints
  \eqref{eq:nrpm-flow-balance}-\eqref{eq:nrpm-non-neg} with
  $\z = F^\M(\x)$ are of the form $\bld{M} \d = f(\x)$ and $\d \geq \bld{0}$
  for some matrix $\bld{M}$ and $f(.)$, a continuously differentiable
  function of $\x$. From this, we can reason that every boundary point
  of $\Mcal{R}_\M$ must belong to $\Mcal{R}_\M$. Thus, $\Mcal{R}_\M$
  and $\Mcal{C}_\M$ are both closed.

  As there are finitely many subsets $\M$ of $\barc$, we have finitely
  many $\Mcal{C}_{\M}$'s which cover $\real^N_+$ and in each
  $\Mcal{C}_{\M}$, $F^\M(.)$ is a continuously differentiable function of
  $\x$. Thus, we can apply Lemma~\ref{lem:gen_dyn_EU}, to conclude
  that for each $\x(0) \in \simplex$, we have existence and uniqueness
  of solutions for NRPM $\forall t \geq 0$.
\end{proofa}

\subsubsection{Proof of Lemma~\ref{lem:nrpm_eqpt}}

\begin{proofa}
  Recall from Lemma \ref{lem:p3_unique} that for any optimizer
  $(\z^*(\x), \d^*(\x))$ of $\prob_3$, $\z^*(\x)$ is unique. Also,
  notice from~\eqref{eq:NRPM_ODE} that if $\x \in \eqpt$ then
  $\z^*(\x) = \x$. Thus, in order to study the equilibrium set, it
  suffices to determine the set of $\x$ for which $\prob_3$ has the
  solution $\z^* = \x$ and $\d^*_{ij} = 0$, $\forall (i,j) \in \arc$
  and $\d^*_{ii} = \x_i$, $\forall i \in \node$. Now, consider an
  arbitrary but fixed $i \in \supp(\x)$. Applying
  \eqref{eq:p_nrpm_kkt_grad2}-\eqref{eq:p_nrpm_kkt_comp_slack2} for
  $(i,i) \in \barc$, we see that $\mu_i^* = \nu_{ii}^* = 0$. Now
  applying
  \eqref{eq:p_nrpm_kkt_grad2}-\eqref{eq:p_nrpm_kkt_comp_slack2} for
  all $(i,j) \in \barc$, we see that
  $\lev_i(\x_i) = \lev_i(\z^*_i) \geq \lev_i(\z_j^*) =
  \lev_i(\x_j)$. This implies that $\x \in \ne$. Using similar
  arguments, we can also show that the converse is true, that is, if
  $\x \in \ne$ then $\z^*(\x) = \x$, $\d^*_{ij} = 0$,
  $\forall (i,j) \in \arc$ and $\d^*_{ii} = \x_i$,
  $\forall i \in \node$ is an optimal solution of $\prob_3$. This
  proves the result.
\end{proofa}

\subsection{Proofs of Auxiliary Results}\label{app:aux-results}

\subsubsection{Proof of Lemma~\ref{lem:gen_dyn_EU}}

\begin{proofa}
  Consider an $\x \in \real^N_+$. Notice that if $\exists$ a
  $\Mcal{C}_i$ such that $\x$ belongs to the interior of $\Mcal{C}_i$
  then $\dyn$ is locally Lipschitz at $\x$ as $\dyn^i$ is locally
  Lipschitz in $\Mcal{C}_i$. The only other possibility is that $\x$
  is a boundary point of $\Mcal{C}_i$ for all
  $i \in \Mcal{I} \subseteq [1,n]_\integer$ and $\x \notin \Mcal{C}_i$
  for all $i \notin \Mcal{I}$. As all the sets $\Mcal{C}_i$ are
  closed, we can therefore find a closed neighborhood
  $\Mcal{B}_2(\x,r)$ around $\x$ such that
  $\Mcal{B}_2(\x, r) \cap \Mcal{C}_i = \varnothing$ for all
  $i \notin \Mcal{I}$.

  Now, for each $i \in \Mcal{I}$ and
  $\forall \y \in \Mcal{B}_2(\x, r) \cap \Mcal{C}_i$,
  \begin{equation*}
    \| \dyn(\bld{y}) - \dyn(\x) \| = \| \dyn^i(\bld{y}) - \dyn^i(\x)
    \| \leq L_{\Mcal{C}_i} \|\bld{y} - \x\| ,
  \end{equation*}
  where $L_{\Mcal{C}_i}$ is a Lipschitz constant for $\dyn^i$ on the
  compact set $\Mcal{B}_2(\x, r) \cap \Mcal{C}_i$. Thus, we can now
  say
  \begin{equation*}
    \| \dyn(\bld{y}) - \dyn(\x) \|  \leq \max_{i \in \Mcal{I}}
    \{L_{\Mcal{C}_i}\} \|\bld{y} - \x\| , \ \forall \y \in
    \Mcal{B}_2(\x, r) ,
  \end{equation*}
  that is, $\dyn$ is locally Lipschitz at $\x$. As $\x \in \real^N_+$
  was arbitrary, $\dyn$ is locally Lipschitz on $\real^N_+$ or on the
  domain $\simplex$.
  
  Now, note that $\simplex$ is a compact subset of $\real_+^{N}$ and
  $\one^T \dot{\x} = \one^T \inmat \,\Delta = \bld{0}$ as the rows of
  $\inmat$ sum to $\bld{0} \in \real^{2M}$. Thus, $\dyn$ is Lipschitz
  in $\simplex$ and $\simplex$ is positively invariant under
  \eqref{eq:gen_dyn}. This guarantees existence and uniqueness of
  solutions $\forall t \geq 0$.
\end{proofa}

\subsubsection{Proof of Lemma~\ref{lem:partial_strict_conv_prob}}

\begin{proofa}
  Let $\bld{q}^1 = (\bld{v}^1, \bld{w}^1)$ and $\bld{q}^2 = (\bld{v}^2,\bld{w}^2)$ be two optimizers with
  $\bld{v}^1 \neq \bld{v}^2$. Since $\bld{q}^1$ and $\bld{q}^2$ are optimizers, it must be that
  \begin{equation*}
    f(\bld{v}^1) = f(\bld{v}^2) .
  \end{equation*}
  Now, consider a convex combination
  $(\bld{v},\bld{w}) \rdef \bld{q} = \sigma \bld{q}^1 + (1-\sigma) \bld{q}^2$ with $\sigma \in (0,1)$. Since
  $\Omega$ is convex, it must be that $\bld{q} \in \Omega$. Then,
  \begin{align*}
    f(\bld{v}) &= f(\sigma \bld{v}^1 + (1-\sigma) \bld{v}^2) \\
    &> \sigma f(\bld{v}^1) + (1-\sigma) f(\bld{v}^2) = f(\bld{v}^1) = f(\bld{v}^2) ,
  \end{align*}
  where the inequality follows from the strict concavity of
  $f(.)$. However, this means that $f(\bld{v}^1) = f(\bld{v}^2)$ is not the
  optimum value and hence $\bld{q}^1$ and $\bld{q}^2$ are not optimizers. This is
  a contradiction and hence the result in the lemma must be true. This completes the proof.
\end{proofa}

\end{document}